\documentclass[a4paper]{amsart}
\usepackage{amssymb}
\usepackage[polutonikogreek, english]{babel}
\usepackage{graphicx}
\usepackage{pstricks}
\usepackage{enumerate}
\theoremstyle{plain}
\newtheorem{theorem}{Theorem}[section]
\newtheorem{lemma}[theorem]{Lemma}
\newtheorem{corollary}[theorem]{Corollary}

\title{Retracing Cantor's first steps\\ in Brouwer's company}
\author{Wim Veldman}
\address{Institute for Mathematics, Astrophysics and Particle Physics, Faculty of Science, Radboud University Nijmegen,
Postbus 9010, 6500 GL Nijmegen, the Netherlands}
\email{W.Veldman@science.ru.nl}
\date{}
\dedicatory{What we call the beginning is often the end\\ And to make an end is to make a beginning.\\The end is where we start from. \\ \flushright T.S.~Eliot, \textit{Little Gidding},  1942}
\begin{document}

\begin{abstract} 

 We prove intuitionistic versions of the classical theorems saying that all countable \textit{closed} subsets of $[-\pi,\pi]$ and even all countable subsets of $[-\pi,\pi]$ are sets of uniqueness. We introduce the \emph{co-derivative} of an open subset  of the set $\mathcal{R}$ of the real numbers as a constructively possibly more useful notion than the derivative of a closed subset of $\mathcal{R}$.  We also have a look at an intuitionistic version of Cantor's theorem that a closed set is the union of a perfect set and an at most countable set.
\end{abstract}
  \maketitle\section{Introduction} G.~Cantor discovered the transfinite and the uncountable while studying and extending B.~Riemann's work on trigonometric series.  He then began  set theory and forgot   his early problems, see \cite{dauben} and \cite{purkert}.

Cantor's work fascinated L.E.J.\;Brouwer but he did not come to terms with it and    started   \textit{intuitionistic mathematics}.  Like Shakespeare, who  wrote his plays as new  versions of works by earlier playwrights, he hoped to turn Cantor's tale into a better story.  

Brouwer  never returned to the questions and results that caused the creation of set theory. Adopting Brouwer's point of view, we do so now.

Brouwer insisted upon a constructive interpretation of statements of the form $A \;\vee\;B$ and $\exists x \in V[A(x)]$. One is only allowed to affirm $A \;\vee\; B$ if one has a reason to affirm $A$ or a reason to affirm $B$. One is only allowed to affirm $\exists x \in V[A(x)]$ if one is able to produce an element $x_0$ of the set $V$ and also evidence for the corresponding statement $A(x_0)$. Brouwer therefore had to reject the logical principle $A\;\vee\;\neg A$.

 Brouwer also came to formulate and accept certain new axioms, most importantly: the \textit{Fan Theorem}, the \textit{Principle of Induction on Monotone Bars} and the \textit{Continuity Principle}.  We shall make use  of these three principles and introduce  them at the  place where we first need them. The first two of them are constructive versions of results in classical, non-intuitionistic analysis. The third one does not stand a non-constructive reading of its quantifiers. It will make its appearance only in Section 10.

The paper has 12 Sections. In Sections 2 and 3, we verify that the two basic results of Riemann that became Cantor's starting point were proven by Riemann in a constructive way.    Section 4 contains an intuitionistic proof of Cantor's Uniqueness Theorem.   The proof requires the Cantor-Schwarz Lemma and this Lemma  obtains two proofs.  Section 5  proves an intuitionistic version of Cantor's result that every finite subset of $[-\pi,\pi]$ is \textit{a set of uniqueness}. In Section 6 we introduce the \textit{co-derivative set} of an an open subset of $\mathcal{R}$, which is itself also an open subset of $\mathcal{R}$. Cantor called a closed subset $\mathcal{F}$ of $[-\pi,\pi]$ \textit{reducible} if one, starting from $\mathcal{F}$, and   iterating the operation of taking the derivative (if needed, transfinitely many times, forming the intersection at limit steps), at last obtains the empty set. We shall call an  open subset $\mathcal{G}$ of $(-\pi,\pi)$ \textit{eventually full} if one, starting from $\mathcal{G}$, and   iterating  the operation of taking the co-derivative (if needed, transfinitely many times, forming the union at limit steps), at last obtains the whole  set $(-\pi,\pi)$. In Section 6  we prove  the intuitionistic version of Cantor's result that every  closed and reducible subset of $[-\pi,\pi]$  is a set of uniqueness: every open and eventually full subset of $(-\pi,\pi)$  \textit{guarantees uniqueness}.  Section 7 gives \textit{examples} of open sets that are eventually full. Section 8 offers an intuitionistic proof, using the Principle of  Induction on Monotone Bars,  that every  co-enumerable  and \textit{co-located} open subset of $(-\pi,\pi)$ is eventually full and, therefore, guarantees uniqueness. Section 9 proves the intuitionistic version of a   stronger theorem, dating from 1911 and due to F.~Bernstein and W.~Young:  all co-enumerable subsets of $[-\pi,\pi]$, not only the ones that are open and co-located, guarantee uniqueness.  This proof needs an extended form of the Cantor-Schwarz Lemma that is proven in two ways. Section 10 shows the simplifying effect of Brouwer's Continuity Principle: it makes Cantor's Uniqueness Theorem trivial and solves Cantor's first and nasty problem in the field of trigonometric expansions, see Lemma \ref{L:cl}(i), in an easy way.   Section 11 treats an intuitionistic version of \textit{Cantor's Main Theorem}: every closed set satisfies one of the alternatives offered by  the continuum hypothesis. In Section 12 we make some observations on Brouwer's  work on Cantor's Main Theorem in \cite{brouwer19}.  
  
 Our journey starts in the middle of the nineteenth century.

\section{Riemann's two results}\label{S:riemann1}

We let $\mathcal{R}$ denote the set of the real numbers. A \textit{real number} $x$ is an infinite sequence $x(0), x(1), \dots$ of pairs $x(n) = \bigl(x'(n), x''(n)\bigr)$ of rationals such that $\forall n[x'(n)\le x'(n+1)\le x''(n+1)\le x''(n)]$  and $\forall m\exists n[x''(n) -x'n <\frac{1}{2^m}]$. For all real numbers $x,y$, we define: $x<_\mathcal{R} y \leftrightarrow \exists n[x''(n) <y'(n)]$ and $x\;\#_\mathcal{R}\;y\leftrightarrow (x<_\mathcal{R} y \;\vee\; y<_\mathcal{R} x)$. The latter \textit{apartness relation} is, in constructive mathematics, more important than the \textit{equality} or \textit{real coincidence} relation.  For all real numbers $x,y$, we define: $x\;\le_\mathcal{R}\;y\leftrightarrow\forall n[x' (n) \le y''(n)]\leftrightarrow \neg (y <_\mathcal{R} x)$ and $x =_\mathcal{R} y \leftrightarrow (x \le_\mathcal{R} y \;\wedge \; y \le_\mathcal{R} x)\leftrightarrow \neg(x \;\#_\mathcal{R}\;y)$. 

It is important that the relations $\le_\mathcal{R}$ and $=_\mathcal{R}$ are \textit{negative} relations. If one wants to prove: $x \le_\mathcal{R} y$ (or: $x=_\mathcal{R}y$, respectively) one may start from the (positive) assumption $y<_\mathcal{R} x$ (or: $x\;\#_\mathcal{R}\; y$, respectively) and try to obtain a contradiction.
 
 We sometimes use the fact that the relations $<_\mathcal{R}$ and $\#_\mathcal{R}$ are \textit{co-transitive}, that is, for all $x,y,z$ in $\mathcal{R}$, $x<_\mathcal{R} z \rightarrow (x<_\mathcal{R} y \;\vee\; y<_\mathcal{R} z)$ and $x\#_\mathcal{R} z \rightarrow (x\#_\mathcal{R} y \;\vee\; y\#_\mathcal{R} z)$.
 
All these relations are, in general,  \textit{undecidable}. For instance, given real numbers $x, y$ one may be unable to say which of the two statements `$x \le_\mathcal{R} y$' or `$y\le_\mathcal{R} x$' is true. Nevertheless, for all numbers $x,y$, one may build a number $z$ such that $x\le_\mathcal{R}z \;\wedge\; y\le_\mathcal{R} z \;\wedge\; \forall u \in \mathcal{R}[(x\le_\mathcal{R}u \;\wedge\; y\le_\mathcal{R} u)\rightarrow z\le_\mathcal{R} u]$. This number, \textit{the least upper bound of $\{x,y\}$}, is denoted by `$\sup(x,y)$'. Similarly, one has $\inf(x,y)$, \textit{the greatest lower bound of $\{x,y\}$}.  

If confusion seems improbable, we sometimes write `$<,\le, =$' where one might expect `$<_\mathcal{R}, \le_\mathcal{R}, =_\mathcal{R}$'. 

\medskip
B. ~Riemann    studied the question: for which functions $F: [-
\pi, \pi]\rightarrow \mathcal{R}$ do there exist real numbers $b_0, a_1, b_1, \ldots$ such that, for all $x$ in $[-\pi, \pi]$,
$$F(x) =\frac{b_0}{2} +\sum_{n>0} a_n \sin nx + b_n \cos nx?$$
He did so in his \textit{Habilitationsschrift} \cite{riemann67}, written in 1854 and published, one year after his death at the age of 39,  by R. ~Dedekind, in 1867. 
Riemann started from a given  infinite sequence of reals
$b_0, a_1, b_1, \ldots$ and  assumed: 
\begin{quote} for each $x$ in $[-\pi, \pi]$, $\lim_{n \rightarrow \infty} (a_n\sin nx + b_n \cos nx )= 0$. \end{quote}

Under this assumption, the function $F$ defined by $$F(x) :=\frac{b_0}{2} +\sum_{n>0} a_n \sin nx + b_n \cos nx$$ 
 is, in general, only a \textit{partial} function from $[-\pi, \pi]$ to $\mathcal{R}$.  Riemann  decided to study the function:

$$G(x) := \frac{1}{4}b_0x^2 +\sum_{n>0}\frac{- a_n}{n^2} \sin nx + \frac{-b_n}{n^2} \cos nx$$

that we obtain by taking, for each term in the infinite sequence defining $F$,  the primitive of the  primitive. Note that, still under the above assumption, $G$ is defined everywhere on $[-\pi,\pi]$.

One might hope that the function $G$ is twice differentiable and that $G'' =F$, but, no,  that  hope seems idle. Riemann decided to replace the second derivative by a symmetric variant. He defined, for all $a,b$ in $\mathcal{R}$ such that $a<b$, for every function $H$ from $[a,b]$ to $\mathcal{R}$, for every $x$ in $(a,b)$,

  $$D^2 H(x) = \lim_{h\rightarrow 0} \frac{H(x+h) + H(x-h) - 2H(x)}{h^2}$$ and $$D^1 H(x) = \lim_{h\rightarrow 0} \frac{H(x+h) + H(x-h) - 2H(x)}{h}$$
 and  proved, for our functions $F$, $G$:
  
   \begin{enumerate}[\upshape 1.] \item for any $x $ in  $(-\pi, \pi)$, if $F(x)$ is defined, then $F(x) = D^2G(x)$, and, \item  for all $x$ in $(-\pi, \pi)$,  $D^1G(x) =0$. \end{enumerate}
   
  Note that,  for all $a,b$ in $\mathcal{R}$ such that $a<b$, for every function $H$ from $[a.b]$ to $\mathcal{R}$, 
    for all $x$ in $(a,b)$, if $H'(x)$ exists, then $D^1H(x)$ exists and $ D^1H(x)=0$.  For assume: $H'(x)$ exists. Then: $$\lim_{h\rightarrow 0} \frac{H(x+h) + H(x-h) - 2H(x)}{h}=$$ $$ \lim_{h \rightarrow 0} \frac{H(x+h)-H(x)}{h} -\lim_{h\rightarrow 0}\frac{H(x)-H(x-h)}{h}=H'(x)-H'(x)=0.$$
    
   Also note that, for any function $H$ from $[a,b]$ to $\mathcal{R}$, for all $x$ in $(a,b)$, if $H''(x)$ exists, then $D^2H(x)$ exists and $H''(x) = D^2H(x)$. The proof of this fact in \cite{bary64}, p. 186, is constructive and we quote it here. Assume: $H''(x)$ exists.
 Then: $$H(x+h) +H(x-h) -2H(x) = \int_0^h H'(x+t) -H'(x-t) dt.$$ Therefore:
 $$\lim_{h\rightarrow 0} |\frac{H(x+h) + H(x-h) - 2H(x)}{h^2} -H''(x)| =$$ $$\lim_{h \rightarrow 0} |\int_0^h\frac{2t}{h^2}\bigl(\frac{H'(x+t) -H'(x-t)}{2t}-H''(x)\bigr)dt|\le$$ $$ \lim_{h\rightarrow 0} \sup_{t\in[0,h]}|\frac{H'(x+t) -H'(x-t)}{2t}-H''(x)|=0.$$ 
 
 Note that, for all functions $H,K$ from $[a,b]$ to $\mathcal{R}$, for every $x$ in $(a,b)$, if $D^1H(x)$ and $D^1K(x)$ both exist, than $D^1(H+K)(x)$ exists and $D^1(H+K)(x) = D^1H(x) + D^1K(x)$, and, similarly, if $D^2H(x)$ and $D^2K(x)$ both exist, than $D^2(H+K)(x)$ exists and $D^2(H+K)(x) = D^2H(x) + D^2K(x)$.
 
 Our final observation will be so useful that we put it into a Lemma.
 
 \begin{lemma}\label{L:basic} Let $a,b$ in $\mathcal{R}$ be given such that $a<b$ and  let $H$ be a function from $[a,b]$ to $\mathcal{R}$. Let $z$ in $(a,b)$ be given such that $D^2H(z)$ is defined. Then: \begin{enumerate}[\upshape (i)] \item if $D^2H(z)>0$ then $\exists y \in[a,b][H(y) > H(z)]$, and, \item if $\forall y \in [a,b][H(y) \le H(z)]$, then $D^2H(z)\le 0$. \end{enumerate}\end{lemma}  
 \begin{proof} (i) Assume: $a<z<b$ and $D^2H(z) >0$. Find $h$ such that $0<h$ and $(z-h, z+h)\subseteq (a, b)$ and $\frac{H(z+h) +H(z-h)-2H(z)}{h^2}=\frac{H(z+h) - H(z)}{h^2} + \frac{H(z-h) - H(z)}{h^2}>0$ and conclude: either $\frac{H(z+h) -H(z)}{h^2}>0$ or $\frac{H(z-h) -H(z)}{h^2}>0$, and therefore, either $H(z) <H(z+h)$ or $H(z) < H(z-h)$, so, in any case: $\exists y \in [a,b][H(z)<H(y)]$. 
 
 \smallskip
 (ii) This follows from (i), by contraposition.\end{proof}
 It is important that the positive statement \ref{L:basic}(i) is behind the negative and often used fact \ref{L:basic}(ii).
 
 \smallskip
 Note that Riemann's first result is already  a partial answer to the problem he tried to solve: \textit{if a function $F$ has a trigonometric expansion on $[-\pi,\pi]$, there must exist a continuous function $G$ from $[-\pi,\pi]$ to $\mathcal{R}$ such that $\forall x \in (-\pi,\pi)[F(x)=D^2G(x)]$}. Riemann draws  more sophisticated conclusions.

   \section{Riemann's constructive proofs of his two basic results }
   
  \subsection{Riemann's first result} We 
   prove, for the functions $F$ and $G$ introduced in Section \ref{S:riemann1}:  \textit{for every $x $ in  $(-\pi, \pi)$, if $F(x)$ is defined, then $F(x) = D^2G(x)$.} 
   
   \smallskip We  follow  Riemann's own argument. 
  
   Note that, for all $x$ in $[-\pi,\pi]$, for all $n>0$, for all $h\neq 0$, 
   $\sin n(x+2h) + \sin n(x-2h) - 2 \sin nx = 2\sin n x(\cos 2nh - 1)= -4\sin nx \sin^2 nh$, and 
    $\cos n(x+2h) + \cos n(x-2h) - 2 \cos nx = 2\cos n x(\cos 2nh - 1)= -4\cos nx \sin^2 nh$, and, therefore: 
    $\frac{\sin n(x+2h) + \sin n(x-2h) - 2 \sin nx}{(2h)^2}\cdot-\frac{1}{n^2} = \sin nx \bigl(\frac{\sin nh} {nh}\bigr)^2$ and

    $\frac{\cos n(x+2h) + \cos n(x-2h) - 2 \cos nx}{(2h)^2}\cdot-\frac{1}{n^2} = \cos nx \bigl(\frac{\sin nh} {nh}\bigr)^2$.
   
   Let $x$ in $(-\pi, \pi)$ be given.  Define $A_0 := \frac{1}{2}b_0$ and, for each $n>0$, $A_n :
   = a_0\sin nx + b_n \cos nx$.  Assume: $F(x) $ is defined, that is: $\sum A_n$ converges. Note: the infinite sequence $A_0, A_1, \ldots$ is bounded and, for all $h$, if $h\neq 0$ and $(x-h,x+h)\subseteq (-\pi,\pi)$, then  $S^0_h:=\frac{G(x+h) + G(x-h) -2G(x)}{h^2}=                                                                                                                                                                                                                                                                                                                                                                                                                                                                                                                                                                                                                                                                                                                                                                                                                                                                                                                                                                                                                                                                                                                                                                                                                                                                                                                                                                                                                                                                                                                                                                                                                                                                                                                                                                                                                                                                                                                                                                                                                                                                                                                                                                                                                                                                                                                                                                                                                                                                                                                                                                                                                                                                                                                                                                                                                                                                                                                                                                                                                             A_0+ \sum_{n>0} A_ n\bigl(\frac{\sin nh}{n h}\bigr)^2$ converges. 
Define, for each $n$, $R_n:=  \sum_{k=n+1}^{\infty} A_k$. 
   
   Let $\varepsilon >0$ be given. Find $N$ such that, for all $m\ge N$, $|R_m| < \frac{1}{6}\varepsilon$.  
   
   Note: for each $n$, $\lim_{h\rightarrow 0} \frac{\sin nh}{nh} = 1$ and find  $h_0$ such that, for each $h<h_0$, $|\sum_{n=0}^{n=N}A_n - \sum_{n=0}^{n=N}A_n \bigl(\frac{\sin nh}{nh}\bigr)^2
   |<\frac{1}{6}\varepsilon$. 
   
   Note:  for each $n$, $R_n - R_{n+1} = A_{n+1}$, and, for each $m$,  \\$S^{m+1}_h:=\sum_{n=m+1}^\infty A_ n\bigl(\frac{\sin nh}{n h}\bigr)^2 = \sum_{n=m}^\infty (R_n-R_{n+1})\bigl(\frac{\sin nh}{n h}\bigr)^2=\\
    R_m\bigl(\frac{\sin mh}{m h}\bigr)^2 +\sum_{n=m+1}^\infty R_n\bigl((\frac{\sin (n+1)h}{(n+1)h})^2-(\frac{\sin nh}{nh})^2  \bigr)$.

   Note:   the function $x \mapsto \frac{\sin(x)}{x}$ is strictly decreasing on the interval $(0, 4]$ and bounded by $1=\lim_{x\rightarrow 0}\frac{sin x}{x}$.
   
   Define $h_1 = \min(h_0, \frac{1}{2}, \frac{3}{N})$. Let $h$ be given such that $0<h<h_1$. Find $M$ in $\mathbb{N}$ such that  $3<Mh <4$. Note: $Nh< Nh_1 \le 3$ and $Mh>3$ and, therefore:  $M>N$. Also:  for all $n$, if  $n<M-1$, then $0<(n+1)h<4$ and  $\frac{\sin nh}{nh} > \frac{\sin (n+1)h}{(n+1)h}$ and $\bigl(\frac{\sin nh}{nh}\bigr)^2 > \bigl(\frac{\sin (n+1)h}{(n+1)h}\bigr)^2$. 
  Conclude: $|\sum_{n= N+1}^{M-1} R_n\bigl((\frac{\sin (n+1)h}{(n+1)h})^2-(\frac{\sin nh}{nh})^2  \bigr)|\le \frac{1}{6}\varepsilon\sum_{n=N+1}^{M-1}\bigl((\frac{\sin nh}{nh})^2-(\frac{\sin (n+1)h}{(n+1)h})^2\bigr)= \frac{1}{6}\varepsilon\bigl((\frac{\sin(N+1)h}{(N+1)h})^2-(\frac{\sin Mh}{Mh})^2\bigr) \le \frac{2}{6}\varepsilon$.
   
  Observe: for each $n$, $|\bigl(\frac{\sin (n+1)h}{(n+1)h}\bigr)^2-\bigl(\frac{\sin nh}{nh}\bigr)^2|\le|\bigl(\frac{\sin (n+1)h}{(n+1)h}\bigr)^2-\bigl(\frac{\sin(n+1)h}{nh}\bigr)^2|+|\bigl(\frac{\sin(n+1)h}{nh}\bigr)^2-\bigl(\frac{\sin nh}{nh}\bigr)^2|\le\frac{1}{h^2}\bigl((\frac{1}{n^2}-\frac{1}{(n+1)^2})+\frac{2h}{n^2}\bigr)$. 
  
  (We are using: for all $x$, $|(\sin (n+1)h)^2 - (\sin nh)^2| = \\|\sin (n+1)h + \sin nh|\cdot|\sin (n+1)h - \sin nh|\le 2h$.) 
  
  Conclude: $\sum_{n= M}^\infty |(\frac{\sin (n+1)h}{(n+1)h})^2-(\frac{\sin nh}{nh})^2  |\le \frac{1}{h^2}(\frac{1}{M^2}+\frac{2h}{M-1})$.  
  
  (We are using: $\sum_{n=M}^\infty \frac{1}{n^2} \le \sum_{n=M}^\infty \frac{1}{n(n-1)}\le\sum_{n=M}^\infty \frac{1}{n-1}-\frac{1}{n} = \frac{1}{M-1}$.)
  
  Conclude: $|\sum_{n= M}^\infty R_n\bigl((\frac{\sin (n+1)h}{(n+1)h})^2-(\frac{\sin nh}{nh})^2  \bigr)|\le  \frac{1}{6}\varepsilon \cdot\frac{1}{h^2}(\frac{1}{M^2}+\frac{2h}{M-1}) =\\ \frac{1}{6}\varepsilon\bigl(\frac{1}{(Mh)^2}+ \frac{2}{(M-1)h}\bigr)\le \frac{1}{6}\varepsilon(\frac{1}{9}+\frac{4}{5})<\frac{1}{6}\varepsilon$.
   
   (We are using: $0<h<\frac{1}{2}$ and $3<Mh<4$, and, therefore: $2\frac{1}{2}<(M-1)h<3\frac{1}{2}$ and $\frac{4}{7}<\frac{2}{(M-1)h}<\frac{4}{5}$.)
   
  Conclude: $|S^0_h - F(x)| \le  |\sum_{n=0}^{n=N}A_n - \sum_{n=0}^{n=N}A_n \bigl(\frac{\sin nh}{nh}\bigr)^2
   |+ |R_N| + |R_N \bigl(\frac{\sin Nh}{Nh}\bigr)^2|+\\ |\sum_{n= N+1}^{M-1}R_n\bigl((\frac{\sin (n+1)h}{(n+1)h})^2-(\frac{\sin nh}{nh})^2  \bigr)| +|\sum_{n=M}^\infty R_n\bigl((\frac{\sin (n+1)h}{(n+1)h})^2-(\frac{\sin nh}{nh})^2  \bigr)|<\frac{6}{6}\varepsilon=\varepsilon$.

   \smallskip
   We thus see: $\forall h<h_1[|S^0_h - F(x) |<\varepsilon]$ and:
   $\forall \varepsilon>0 \exists k\forall h<k[|S^0_h - F(x) |<\varepsilon]$, that is: $D^2G(x) = \lim_{h\rightarrow 0}S^0_h=F(x)$.
    \subsection{Riemann's second result} We 
   prove, for the function $G$ introduced in Section \ref{S:riemann1}:   \textit{for every $x $ in  $(-\pi, \pi)$, $ D^1G(x)=0$.}
   
   \smallskip
    Let $x$ in $(-\pi, \pi)$ be given.  Define $A_0 := \frac{1}{2}b_0$ and, for each $n>0$, $A_n :
   = a_0\sin nx + b_n \cos nx$.  Riemann's assumption implies: $\lim_{n\rightarrow \infty}A_n = 0$.  Therefore, the infinite sequence  $A_0, A_1, \ldots$ is bounded and, for all $h\neq 0$, 
    $S^0_h:=\frac{G(x+h) + G(x-h) -2G(x)}{h^2}=                                                                                                                                                                                                                                                                                                                                                                                                                                                                                                                                                                                                                                                                                                                                                                                                                                                                                                                                                                                                                                                                                                                                                                                                                                                                                                                                                                                                                                                                                                                                                                                                                                                                                                                                                                                                                                                                                                                                                                                                                                                                                                                                                                                                                                                                                                                                                                                                                                                                                                                                                                                                                                                                                                                                                                                                                                                                                                                                                                                                                              \sum A_ n\bigl(\frac{\sin nh}{n h}\bigr)^2$ converges. Note: $D^1G(x) = \lim_{h\rightarrow 0} hS^0_h$. 
    
   Let $\varepsilon>0$ be given.  Find $N$ such that, for all $n\ge N$, $|A_n |<\frac{1}{5}\varepsilon$.  Define $Q:=  \sum_{n=1}^N |A_n|$.  
  Let $h$ be given such that $0<h<\frac{1}{2}$ and $h< \frac{3}{N}$. Find $M$ such that $3< Mh < 4$.    Note: $M>N$, and:
   $|S_h^0|\le  \sum_{n=0}^N|A_n\bigl(\frac{\sin nh}{n h}\bigr)^2| + \sum_{n=N+1}^M|A_n\bigl(\frac{\sin nh}{n h}\bigr)^2|+\sum_{n=M+1}^\infty|A_n\bigl(\frac{\sin nh}{n h}\bigr)^2|\le Q + M \frac{1}{5}\varepsilon+ \sum_{n=M+1}^\infty |A_n|\bigl(\frac{1}{n h}\bigr)^2\le Q + M \frac{1}{5}\varepsilon+ \frac{1}{h^2}\frac{1}{5}\varepsilon\sum_{n=M+1}^\infty\frac{1}{n^2}\le Q + M \frac{1}{5}\varepsilon+\frac{1}{M}\frac{1}{5h^2}\varepsilon\le Q + \frac{4}{h}\frac{1}{5}\varepsilon + \frac{1}{3h}\frac{1}{5} \varepsilon$, and $|hS_h^0|\le Qh +\frac{13}{15}\varepsilon$. Conclude: for all $h$, if $0<h<\frac{2}{15}\frac{\varepsilon}{Q}$, then $|hS^0_h| < \varepsilon$. Conclude: $D^1G(x)=\lim_{h \rightarrow 0} hS^0_h = 0$.
  \section{Cantor's Uniqueness Theorem}
 \subsection{An  intuitionistic proof of the Cantor-Schwarz-lemma} 
  Cantor, after studying Riemann's \textit{Habilitationsschrift} \cite{riemann67}, tried to prove: for every function $F$ from $[-\pi,\pi]$ to $\mathcal{R}$ there exists \textit{at most one} infinite sequence of reals $b_0, a_0, b_1, \ldots$  such that, for all $x$ in $[-\pi, \pi]$, $F(x) =\frac{b_0}{2} +\sum_{n>0} a_n \sin nx + b_n \cos nx.$
  
  He quickly saw that this statement is equivalent to the statement: for every infinite sequence of reals $b_0, a_0, b_1, \ldots$, if,  for all $x$ in $[-\pi, \pi]$, $0 =\frac{b_0}{2} +\sum_{n>0} a_n \sin nx + b_n \cos nx$, then $b_0=0$ and, for all $n>0$, $a_n = b_n =0$.
  
  Cantor needed the following lemma. The proof is due to 
   H.A.~Schwarz, see \cite{cantor70}. The intuitionistic proof we give now is a nontrivial elaboration of the argument given by Schwarz and uses the \textit{Fan Theorem}. We found some  inspiration for this proof in the classical proof of a Lemma we shall consider later in this paper, Lemma \ref{L:csby}.
  
   The Fan Theorem appears in \cite{brouwer24} and \cite{brouwer27}. We now give a brief exposition.
   
   We let $Bin$ denote the set of all finite  sequences $a$ such that
   $\forall i<\mathit{length}(a)[a(i)=0 \;\vee \; a(i)=1]$. The elements of $Bin$ are called finite \textit{binary} sequences. Cantor space $\mathcal{C}$ is the set of all infinite sequences $\alpha$ such that $\forall i[\alpha(i) =0\;\vee\;\alpha(i) = 1]$.  For all $\alpha$ in $\mathcal{N}$, for all $n$, we define: $\overline \alpha n :=\bigl(\alpha(0), \alpha(1),  \ldots \alpha(n-1)\bigr)$, the \textit{initial part} of $\alpha$ of length $n$. A subset $B$ of $Bin$ is called a \textit{bar in $\mathcal{C}$} if and only if $\forall \alpha \in \mathcal{C}\exists n[\overline \alpha n \in B]$. The \textit{(unrestricted)}\footnote{One sometimes \textit{restricts} the Fan Theorem to \textit{decidable} subsets of $\mathit{Bin}$. Brouwer's argument seems to establish the unrestricted version, and, in any case, the unrestricted version follows from the restricted version with the help of a strong form of  Brouwer's Continuity Principle, see Section \ref{S:bcp}.} Fan Theorem is the statement that every subset of $Bin$ that is a bar in $\mathcal{C}$ has a finite subset that is a bar in $\mathcal{C}$. The argument Brouwer gives   is philosophical rather than mathematical.
  
    The Fan Theorem fails to be true in \textit{recursive} or \textit{computable} mathematics, see \cite{veldman11}.
   \begin{lemma}[Cantor-Schwarz]\label{L:cs} Let $a<b$ in $\mathcal{R}$ be given and let $G$ be a continuous function from $[a,b]$ to $\mathcal{R}$ such that $G(a) = G(b) =0$ and,  for all $x$ in $(a,b)$, $D^2G(x)$ exists. \begin{enumerate}[\upshape (i)] \item For each $\varepsilon >0$, if $\exists x \in [a,b][G(x)=\varepsilon]$, then $\exists z \in(a,b)[D^2G(z)\le -2\varepsilon]$, and,  \item if $\forall x \in (a,b)[D^2G(x) = 0]$, then $\forall x \in [a,b][G(x) = 0]$. \end{enumerate} \end{lemma}
   \begin{proof}(i) Assume we find $x$ in $[a,b]$ such that $G(x) =\varepsilon>0$. Let $H$ be the function from $[a,b]$ to $\mathcal{R}$ such that, for all $y$ in $[a,b]$, $H(y)= G(y) -\varepsilon\frac{(b-y)(y-a)}{(b-a)^2}$. Note: $H(x) \ge \frac{3}{4}\varepsilon$ and, for all $y$ in $(a,b)$, $D^2H(y)$ exists and $D^2H(y) = D^2G(y) +2\varepsilon$. 
   
   Cantor, upon the advice of Schwarz, now used  Weierstrass's result that a continuous function defined on a closed interval assumes at some point its greatest value and produced $z$ in $[a,b]$ such that, for all $y$ in $[a,b]$, $H(y) \le H(z)$. The constructive mathematician knows that such a $z$ may not always be found and  has to adapt the  argument.
   
   We use the fact that, for all $a,b$ in $\mathcal{R}$ such that $a<b$, for every continuous function $f$ from $[a,b]$ to $\mathcal{R}$, one may find $s$ in $\mathcal{R}$ such that (i) for all $y$ in $[a,b]$, $f(y) \le s$, and (ii) for every $\varepsilon >0$  there exists $y$ in $[a,b]$ such that $f(y) > s-\varepsilon$. This number $s$ is the \textit{supremum} or \textit{least upper bound} of the set $\{f(y)|y\in [a,b]\}$, notation: $\sup_{[a,b]}(f)$. The fact that  this number exists is a consequence of the (restricted) Fan Theorem, see \cite{veldman11}.
   
   For each real number $\rho$ we 
   let $H_\rho$ be the function from $[a,b]$ to $\mathcal{R}$ such that, for all $y$ in $[a,b]$, $H_\rho(y) = H(y) + \rho \frac{y-a}{b-a}$.  Note: for each $\rho$, for all $y$ in $(a,b)$, $D^2H_\rho(y)$ exists  and $D^2H_\rho(y) =D^2H(y) = D^2G(y) + 2\varepsilon$. Also note: for all $\rho$ in $[0, \frac{1}{2}\varepsilon]$, $H_\rho(x) =H(x) +\rho\frac{x-a}{b-a}\ge \frac{3}{4}\varepsilon >0= H_\rho(a)$ and $H_\rho(x)\ge\frac{3}{4}\varepsilon > \rho=H_\rho(b)$. 
   
   We now simultaneously construct $\rho$ in $[0, \frac{1}{2}\varepsilon]$ and $z$ in $[a,b]$ such that, for all $y$ in $[a,b]$,  if $y \;\#_\mathcal{R}\; z$, then $H_\rho(y) < H_\rho(z)$.

   To this end, we define two infinite sequences of pairs of reals, $(a_0, b_0), (a_1, b_1), \ldots$ and $(c_0, d_0), (c_1, d_1), \ldots$ and an infinite sequence $\delta_0, \delta_1,\ldots$ of reals such that 
 $(a_0 ,b_0) =(a, b)$  and  $(c_0 ,d_0) = (0,\frac{1}{2}\varepsilon)$, and, for each $n$, \begin{enumerate}[\upshape (i)]   
\item  $a_n\le a_{n+1}<b_{n+1}\le b_n$ and  $c_n\le c_{n+1}<d_{n+1}\le d_n$, and, \item  $\delta_n=\frac{1}{81}(b_n-a_n)(d_n-c_n)$, and, \item for each $\rho$ in $[c_{n+1}, d_{n+1}]$, $\sup_{[a_{n+1}, b_{n+1}]}H_\rho = \sup_{[a_{n}, b_{n}]}H_\rho$  and, 
 for each $y$ in $[a,b]$, if $y \notin [a_{n+1}, b_{n+1}]$, then $H_\rho(y) +\delta_n\le \sup_{[a_{n+1}, b_{n+1}]}H_\rho $.
\end{enumerate}

 We do so as follows. Suppose: $n \in \mathbb{N}$ and we defined $a_n, b_n, c_n, d_n$.  Define $z_0:= \frac{2}{3}a_n + \frac{1}{3}b_n$ and $z_1:= \frac{1}{3}a_n + \frac{2}{3}b_n$ and   consider $[a_n, z_0]$ and $[z_1, b_n]$, the \textit{left third part} and the \textit{right third part} of $[a,b]$. Define $\rho_0 := \frac{2}{3} c_n + \frac{1}{3}d_n$ and $\rho_1 := \frac{1}{3} c_n + \frac{2}{3}d_n$. 
 
  Note: $\rho_0<\rho_1$, and,    for all $x$ in $[a,b]$, $H_{\rho_0}(x) +(\rho_1-\rho_0)x = H_{\rho_1}(x)$, and, therefore, $$\sup_{[a_n, z_0]}(H_{\rho_1}) \le \sup_{[a_n, z_0]}(H_{\rho_0})+ z_0(\rho_1 -\rho_0),$$ and also: $$\sup_{[z_1, b_n]}(H_{\rho_1}) \ge \sup_{[z_1, b_n]}(H_{\rho_0})+z_1(\rho_1-\rho_0).$$ Conclude:  $$\sup_{[z_1, b_n]}(H_{\rho_1})-\sup_{[a_n, z_0]}(H_{\rho_1}) \ge \sup_{[z_1, b_n]}(H_{\rho_0})-\sup_{[a_n, z_0]}(H_{\rho_0})+(z_1-z_0)(\rho_1 -\rho_0).$$ Define $\delta_n := \frac{1}{9}(z_1-z_0)(\rho_1-\rho_0)$. Then: $$\bigl(\sup_{[z_1, b_n]}(H_{\rho_1})-\sup_{[a_n, z_0]}(H_{\rho_1})\bigr)+ \bigl( \sup_{[a_n, z_0]}(H_{\rho_0})-\sup_{[z_1, b_n]}(H_{\rho_0})\bigr)> 6\delta_n,$$ and, therefore, either: $$\sup_{[z_1, b_n]}(H_{\rho_1})-\sup_{[a_n, z_0]}(H_{\rho_1})>3\delta_n$$ or: $$\sup_{[a_n,z_0]}(H_{\rho_0})-\sup_{[z_1,b_n]}(H_{\rho_0})>3\delta_n.$$ We thus have two cases.
 
   \textit{Case (i)}. $\sup_{[z_1, b_n]}(H_{\rho_1})>\sup_{[a_n, z_0]}(H_{\rho_1})+3\delta_n$. Define $(a_{n+1} ,b_{n+1}) :=(z_0, b_n)$ and $(c_{n+1} ,d_{n+1}) := \bigl(\sup(\rho_0,\rho_1 -\delta_n), \inf(d_n, \rho_1 +\delta_n)\bigr)$. 
   
   Note: for each $\rho$ in  $[c_{n+1}, d_{n+1}]$,  $\forall y \in [a,b][H_{\rho_1}(y) - H_\rho(y)|\le |\rho_1 -\rho|\frac{y-a}{b-a}]$ and $|\sup_{[z_1, b_n]}(H_{\rho_1})-\sup_{[z_1, b_n]}(H_{\rho})|\le\delta_n$ and $|\sup_{[a_n, z_0]}(H_{\rho_1})-\sup_{[a_n, z_0]}(H_{\rho})|\le\delta_n$, and, therefore: $\sup_{[z_1, b_n]}(H_{\rho})>\sup_{[a_n, z_0]}(H_{\rho})+\delta_n$ and: $\sup_{[a_{n+1}, b_{n+1}]}H_\rho =\sup_{[z_0, b_n]}H_\rho = \sup_{[a_{n}, b_{n}]}H_\rho$.
   
   Also note: for each $\rho$ in  $[c_{n+1}, d_{n+1}]$, for each $y$ in $[a,b]$, if $y \notin [a_{n+1}, b_{n+1}]$, then $\exists z \in [a_{n+1}, b_{n+1}][H_\rho(y) +\delta_n< H_\rho(z)]$.
   \smallskip
   
   \textit{Case (ii)}. $\sup_{[z_1, b_n]}(H_{\rho_0})+3\delta_n <\sup_{[a_n, z_0]}(H_{\rho_0})$. Define $(a_{n+1} ,b_{n+1}) :=(a_n,z_1)$ and $(c_{n+1} ,d_{n+1}) := \bigl(\sup(c_n,\rho_0 -\delta_n), \inf(\rho_1, \rho_0 +\delta_n)\bigr)$. 
   
    Note: for each $\rho$ in  $[c_{n+1}, d_{n+1}]$, $\forall y \in [a,b][H_{\rho_0}(y) - H_\rho(y)|\le |\rho_0 -\rho|\frac{y-a}{b-a}]$ and $|\sup_{[z_1, b_n]}(H_{\rho_0})-\sup_{[z_1, b_n]}(H_{\rho})|\le\delta_n$ and $|\sup_{[a_n, z_0]}(H_{\rho_0})-\sup_{[a_n, z_0]}(H_{\rho})|\le\delta_n$, and, therefore: $\sup_{[z_1, b_n]}(H_{\rho})+\delta_n<\sup_{[a_n, z_0]}(H_{\rho})$ and: $\sup_{[a_{n+1}, b_{n+1}]}H_\rho =\sup_{[a_n,z_1]}H_\rho = \sup_{[a_{n}, b_{n}]}H_\rho$.
    
     Again: for each $\rho$ in  $[c_{n+1}, d_{n+1}]$, for each $y$ in $[a,b]$, if $y \notin [a_{n+1}, b_{n+1}]$, then $\exists z \in [a_{n+1}, b_{n+1}][H_\rho(y)+\delta_n < H_\rho(z)]$.
    \smallskip
   
   Note that  \textit{Case (i)} does not exclude case \textit{Case (ii)}. We have to make  choices. One may provide details as to how to make these choices in terms of $a, b$ and $\varepsilon$, real numbers that are given to us as infinite sequences of rational approximations. There is no need to apply an Axiom of Countable Choice.

   This completes the description of our construction. Note: for each $n$, $b_{n+1} -a_{n+1} = \frac{2}{3}(b_n -a_n)$ and $d_{n+1} -c_{n+1} \le \frac{2}{3}(d_n -c_n)$. Applying the Cantor Intersection Theorem, we find $\rho,z$ such that, for all $n$, $c_n \le \rho \le d_n$ and  $a_n \le z \le b_n$.   Assume: $y \in [a,b]$ and $y \;\#_\mathcal{R} \; z$.  Find $n$ such that $y \notin[a_n, b_n]$ and conclude: $H_\rho(y) +\delta_n \le H_\rho(z)$. The function $H_\rho$ thus assumes its greatest value at $z$.
 
 As we observed earlier,   $ H_\rho(a) < H_\rho(x) \le H_\rho(z)$ and  also $H_\rho(b) < H_\rho(x) \le H_\rho(z)$, and, therefore\footnote{For this step, see Lemma \ref{L:helphelpcsby}(i).}:  $a<z<b$ and $D^2H_\rho(z)$ exists.  Use Lemma \ref{L:basic}(ii) and conclude:  $D^2(H_\rho)(z) \le 0$, and:  $D^2G(z) =D^2(H_\rho)(z)-2\varepsilon\le -2\varepsilon$. 
  
   \smallskip (ii) Consider the function $-G$. Applying (i),   conclude:   $$\forall\varepsilon >0[\exists x \in [a,b][G(x)=-\varepsilon] \rightarrow\exists z \in(a,b)[D^2G(z) \ge 2\varepsilon]].$$ Assume $\forall x \in (a,b)[D^2G(x) = 0]$, and conclude $$\forall \varepsilon>0\neg\exists x \in [a,b][G(x) =\varepsilon \;\vee\;G(x) = -\varepsilon]],$$ that is:  $\forall x \in [a,b][G(x) = 0]$.\end{proof}
   \begin{corollary}\label{cor:cs} Let $a,b$ be real numbers such that $a<b$  and let $G$ be a function from $[a,b]$ to $\mathcal{R}$ such that,  for all $x$ in $(a,b)$, $D^2G(x)=0$. Then $G$ is linear on $[a,b]$, that is: for all $x$ in $[a,b]$, $G(x)=G(a) +\frac{x-a}{b-a}\bigl(G(b)-G(a)\bigr)$. \end{corollary} 
  
   \begin{proof} Define, for each $x$ in $[a,b]$:  $$G^\ast(x):= G(x)-G(a) -\frac{x-a}{b-a}\bigl(G(b)-G(a)\bigr)$$ and conclude, using Lemma \ref{L:cs}(ii): for all $x$ in $[a,b]$, $G^\ast(x) = 0$. \end{proof}   
\subsection{A second proof of the Cantor-Schwarz-Lemma}\subsubsection{}\label{SSS:hb} Let $a,b$ in $\mathcal{R}$ be given such that $a<b$. We define a function $D$ associating to every $s$ in $Bin$ a pair $D(s) =\bigl(D_0(s), D_1(s)\bigr)$ of real numbers, such that 
    $D(( \;)) := (a,b)$ and, for all $s \in Bin$,  $D(s\ast( 0))=\bigl(D_0(s), \frac{1}{3}D_0(s) + \frac{2}{3}D_1(s)\bigr)$ and $D(s\ast( 1))=\bigl(\frac{2}{3}D_0(s) + \frac{1}{3}D_1(s), D_1(s)\bigr)$.  
 Let $\phi$ be a function from Cantor space $\mathcal{C}$ to $[a,b]$ such that, for every $\alpha$ in $\mathcal{C}$, for every $n$, $D_0(\overline \alpha n)<\phi(\alpha)<D_1(\overline \alpha n)$. Note that $\phi$ is a surjective map from $\mathcal{C}$ onto $[a,b]$. Using this function and applying the unrestricted Fan Theorem, one may prove
   the following statement,   the \textit{unrestricted Heine-Borel Theorem}: \begin{quote}{\it Let $a,b$ in $\mathcal{R}$ be given such that $a<b$. Let $\mathcal{B}$ be a subset of $\mathcal{R}^2$ such that  $\forall x \in [a,b]\exists ( c, d) \in \mathcal{B}[c<x<d]$. Then  there exist $n$ in $\mathbb{N}$, $( c_0, d_0), ( c_1, d_1), \ldots, ( c_{n}, d_{n})  $ in $\mathcal{B}$ such that $\forall x \in [a,b]\exists i\le n [c_i<x<d_i]$.}\end{quote}

The proof is as follows. Let $B$ be the set of all $a$ in $Bin$ such that, for some $(c,d)$ in $\mathcal{B}$, $c<D_0(a) < D_1(a)<d$. Note that $B$ is a bar in Cantor space $\mathcal{C}$. Find a finite subset of $B$ that is bar in $\mathcal{C}$ and enumerate its elements: $a_0, a_1, \ldots a_n$. Then find   $( c_0, d_0), ( c_1, d_1), \ldots, ( c_{n}, d_{n})  $ in $\mathcal{B}$ such that, for each $i\le n$, $c_i<D_0(a_i) < D_1(a_i)<d_i$. This finite sequence of elements of $\mathcal{B}$ satisfies the requirements.
 
 In \cite{veldman11}, the restricted Heine-Borel Theorem is derived from the restricted Fan theorem. 
  
  \subsubsection{} Weierstrass's Theorem, saying  that a continuous function defined on a closed interval assumes at some point its greatest value, fails constructively. We have the following negative substitute:
  
  \begin{theorem}\label{T:weierstr} Let $a<b$ in $\mathcal{R}$ be given and let $H$ be a continuous function from $[a,b]$ to $\mathcal{R}$. Then $\neg\forall z \in[a,b]\exists y \in[a,b][H(z) <H(y)]$. \end{theorem}
  
  \begin{proof} Assume: $\forall z \in [a,b]\exists y \in [a,b][H(z) < H(y)]$. Let $\mathcal{B}$ be the set of all  pairs $(c,d)$ of reals such that $\exists z\in(a,b)[c<z<d]$  and $\exists y \in [a,b]\forall v \in (c,d)\cap [a,b][H(v)<H(y)]$.  Assume: $z\in [a,b]$. Find $y \in [a,b]$ such that $H(z) <H(y)$. Using the continuity of $H$, find $c,d$ such that $c<z<d$ and $\forall v \in (c,d)\cap [a,b][H(v)<H(y)]$. Note: $(c,d) \in \mathcal{B}$. We thus see: $\forall z \in [a,b]\exists (c,d) \in \mathcal{B}[c<z<d]$.
  
  Applying the Heine-Borel Theorem, find  $n$ in $\mathbb{N}$, $( c_0, d_0), ( c_1, d_1), \ldots, ( c_{n}, d_{n})  $ in $\mathcal{B}$ such that $\forall z \in [a,b]\exists i\le n [c_i<z<d_i]$.
  
  Define a binary relation $<_\ast$ on the set $\{0,1,\ldots,n\}$ such that, for all $i,j\le n,$ $$i<_\ast j \leftrightarrow \exists y \in (c_j,d_j)\cap[a,b]\forall z \in(c_i,d_i)\cap[a,b][H(z) <H(y)].$$
  
  Note: $\forall i\le n[\neg(i <_\ast i)]$ and $\forall i\le n\forall j \le n\forall k \le n[(i <_\ast j \;\wedge\; j<_\ast k)\rightarrow i <_\ast k]$ and $\forall i \le n \exists j \le n[ i<_\ast j]$. 
  
  That is impossible. Conclude: $\neg\forall z \in [a,b]\exists y \in [a,b][H(z) < H(y)]$.\end{proof}
\subsubsection{} The  proof\footnote{The proof of the Cantor-Schwarz Lemma in \cite{veldman03} is wrong.} of the Cantor-Schwarz Lemma we now are to set forth is shorter than the first one but, from a constructive point of view, it is less informative, as the conclusion of Lemma \ref{L:cs}(i) is not obtained. \begin{lemma}[Cantor-Schwarz, version II]\label{L:cs2} Let $a<b$ in $\mathcal{R}$ be given and let $G$ be a continuous function from $[a,b]$ to $\mathcal{R}$ such that $G(a) = G(b) =0$. \\If $\forall x \in (a,b)[D^2G(x)= 0]$, then $\forall x \in [a,b][G(x) = 0]$.  \end{lemma}
   \begin{proof}(i) Assume we find $x$ in $[a,b]$ such that $G(x) >0$. Define a function $H$ from $[a,b]$ to $\mathcal{R}$ such that, for all $y$ in $[a,b]$, $H(y)= G(y) -G(x)\frac{(b-y)(y-a)}{(b-a)^2}$. Note: $H(x) \ge \frac{3}{4}G(x) >0$ and, for all $y$ in $(a,b)$, $D^2H(y)$ exists and $D^2H(y) = D^2G(y) +2G(x)>0$. We are going to show that every point in $[a,b]$ positively refuses to be the point where $H$ assumes its greatest value, and this, according to Theorem \ref{T:weierstr}, is impossible.

First, use the fact that $H$ is continuous at $a$ and at $b$ and that $H(a) = H(b) =0$ and find $\delta>0$ such that  $\forall z \in [a,a+\delta)\cup(b-\delta, b][H(z)<H(x)]$. 
Then, assume: $z\in(a, b)$. Note: $D^2H(z) >0$. Use Lemma \ref{L:basic} and conclude: $\exists y \in [a,b][H(z)<H(y)]$. 
Observe: $[a,b] = [a,a+\delta)\cup(a, b)\cup(b-\delta, b]$, and  conclude: $\forall z \in [a,b]\exists y \in[a,b][H(z) < H(y)]$.  Contradiction, according to Theorem \ref{T:weierstr}. 

Therefore: $\neg \exists x \in [a,b][G(x) >0]$. 

One may prove in the same way: $\neg \exists x \in [a,b][G(x) <0]$ and conclude: $\forall x \in[a,b][G(x) = 0]$. \end{proof}

   \subsection{The Uniqueness Theorem} \begin{theorem}[Cantor's Uniqueness Theorem, see \cite{cantor70}]\label{T:cu} For every infinite sequence of reals $b_0, a_1, b_1, \ldots$, if, for all $x$ in $[-\pi,\pi]$, $$F(x): = \frac{b_0}{2}+\sum_{n>0} a_n \sin nx + b_n \cos nx =0,$$ then $b_0=0$ and, for all $n>0$, $a_n = b_n =0$. \end{theorem} 
   
   \begin{proof} Assume: the infinite sequence of reals $b_0, a_1, b_1, \ldots$ satisfies the conditions of the theorem. Then also, for each $x$ in $[-\pi,\pi]$, $\lim_{n\rightarrow \infty} a_n\sin nx + b_n \cos nx = 0$, that is: Riemann's condition is satisfied. Taking $x=0$, we also see: $\lim_{n\rightarrow \infty} b_n =0$.
   
   \smallskip
   
   \textit{We make a provisional assumption}: the infinite sequence  $b_0, a_1, b_1, \ldots$ is bounded.
   
   Consider $$G(x) := \frac{1}{4}b_0x^2+\sum_{n>0} \frac{-a_n}{n^2} \sin nx + \frac{-b_n}{n^2} \cos nx. $$ According to Riemann's first result, for all $x$ in $(-\pi,\pi)$, $D^2G(x) =0$. Use Corollary \ref{cor:cs} and find $a, b$ such that, for all $x$ in $[-\pi,\pi]$, $G(x) = ax + b$. 
   
   Note: $G(\pi)=G(-\pi)$ and conclude: $a=0$ and, for all $x$ in $[-\pi,\pi]$, $G(x) =b$.
   
   Note: for each $n>0$, $\frac{1}{\pi}\int_{-\pi}^\pi x^2 \cos nx dx = \frac{-1}{n\pi}\int_{-\pi}^\pi  2x \sin nx dx=\frac{2x}{n^2\pi}\cos nx \mid^\pi_{-\pi}= (-1)^{n}\frac{4}{n^2}$, and, for each $n>0$, $\frac{1}{\pi}\int_{-\pi}^\pi x^2 \sin nx dx =0$.
   
   Define, for each $N>0$ in $\mathbb{N}$, $G_N(x) := \frac{1}{4}b_0x^2+\sum_{n=1}^N \frac{-a_n}{n^2} \sin nx + \frac{-b_n}{n^2} \cos nx$. 
   
   As the infinite sequence $b_0, a_1, b_1, \dots$ is bounded, the sequence of functions $G_1, G_2, \ldots$ converges to the function $G$ \textit{uniformly on $[-\pi,\pi]$}.
   
    Conclude: for each $n>0$, $0=\frac{1}{\pi}\int_{-\pi}^\pi b \sin nxdx=\frac{1}{\pi}\int_{-\pi}^\pi G(x) \sin nxdx= \\\lim_{N\rightarrow \infty}\frac{1}{\pi}\int_{-\pi}^\pi G_N(x) \sin nxdx=\frac{-a_n}{n^2}$, so $a_n=0 $. 
   
   Also conclude: for each $n>0$, $0=\frac{1}{\pi}\int_{-\pi}^\pi b \cos nxdx=\frac{1}{\pi}\int_{-\pi}^\pi G(x) \cos nxdx= \\\lim_{N\rightarrow \infty}\frac{1}{\pi}\int_{-\pi}^\pi G_N(x) \cos nxdx=\frac{(-1)^nb_0-b_n}{n^2}$, so $b_n=(-1)^nb_0 $. As $\lim_{n \rightarrow \infty}b_n =0$, we conclude: for all $n$, $b_n =0$. 
   
   \smallskip \textit{One may do without the provisional assumption.} We no longer assume that  the infinite sequence $b_0, a_1, b_1, \ldots$ is bounded.
   
   Using a suggestion made by Riemann himself in par. 12 of \cite{riemann67} and repeated by L.~Kronecker in a letter to Cantor, see the footnote in \cite{cantor71}, we continue as follows.
   
    Assume: $x \in [-\pi,\pi]$. Define for each $t$ in $[-\pi, \pi]$, $$ K(t) = F(x+t) + F(x-t)$$ and note: $$K(t) = b_0 + 2\sum_{n>0}(a_n\sin nx + b_n \cos nx)\cos nt,$$
    
    as, for each $n$, for each $t$, $\sin(nx + nt) + \sin(nx - nt) = 2\sin nx \cos nt$ and $\cos(nx + nt) + \cos(nx - nt) = 2\cos nx \cos nt$.
    Note: for each $t$ in $[-\pi,\pi]$, $K(t)=0$, and: the sequence $n \mapsto a_n\sin nx + b_n \cos nx$ converges and is bounded. Using the first part of the proof,  conclude: $b_0 = 0$ and, for each $n>0$,  $ a_n\sin nx + b_n \cos nx=0$. This conclusion holds for all $x$ in $[-\pi,\pi]$.  Conclude: $b_0 = 0$ and, for each $n>0$, $a_n = b_n = 0$.  \end{proof}

   \section{Every co-finite subset of $[-\pi,\pi]$ guarantees uniqueness}  Let $\mathcal{X}$ be a subset of $[-\pi,\pi]$. We define: \textit{$\mathcal{X}$ guarantees uniqueness}\footnote{  In classical mathematics, one uses the expression: \textit{$[-\pi,\pi]\setminus\mathcal{X}$ is a set of uniqueness} where we say: \textit{$\mathcal{X}$ guarantees uniqueness}.} if and only if, for every infinite sequence of reals $b_0, a_1, b_1, \ldots$, if for all $x$ in $\mathcal{X
   }$, $F(x) := \frac{b_0}{2}+\sum_{n>0} a_n \sin nx + b_n \cos nx =0$, then $b_0=0$ and, for all $n>0$, $a_n = b_n =0$.

   \smallskip
   We define: a subset $\mathcal{X}$ of $\mathcal{R}$ is \textit{co-finite} if and only if there exist $n$ in $\mathbb{N}$, $x_0, x_1, \ldots, x_{n-1}$ in $\mathcal{R}$ such that, for all $x$ in $\mathcal{R}$, if $\forall i<n[x\;\#_\mathcal{R}\; x_i]$, then $x \in \mathcal{X}$. 
   
   A set of the form $\{x_0, x_1, \ldots, x_{n-1}\}$ is called a \textit{finitely enumerable} set of real numbers. As the relation $=_\mathcal{R}$ of real coincidence is not a decidable relation, we often are unable to determine the number of elements of a finitely enumerable set of reals; that is, constructively, a finitely enumerable set of reals is not necessarily a \textit{finite} set of reals. 
   
   \medskip
   Cantor saw, how to prove, with the help of Riemann's second result, the following extension of his Uniqueness Theorem, see \cite{cantor71}.
   
   \begin{theorem}\label{T:cucofinite} Every co-finite subset of $[-\pi,\pi]$ guarantees uniqueness.\end{theorem}
   
   \begin{proof} Find $n$ in $\mathbb{N}$ and  $x_0, x_1, \ldots,x_{n-1}$ in $\mathcal{R}$ such that, for all $x$ in $[-\pi,\pi]$, if $\forall i<n[x\;\#_\mathcal{R}\; x_i]$, then $x \in \mathcal{X}$.
   
   The numbers $x_0, x_1, \ldots,x_{n-1}$ may not belong to $[-\pi,\pi]$. We solve this little problem as follows.  Define, for each $i<n$, $y_i :=\inf (\pi, \sup (x_i, -\pi)\bigr)$. Note: for each $i<n$, $y_i \in [-\pi,\pi]$ and, for all $x$ in $[-\pi,\pi]$, $\forall i<n[ x \;\#_\mathcal{R}\; x_i]$ if and only if   $\forall i<n[ x \;\#_\mathcal{R}\; y_i]$. Define $\mathcal{Y}:=\{\pi,-\pi,y_0,y_1,\ldots,y_{n-1}\}$.

  We make a  \textit{provisional assumption: $\forall x \in \mathcal{Y}\forall y \in \mathcal{Y}[x=_\mathcal{R} y \;\vee\; x\;\#_\mathcal{R}\;y]$.}
   
   We then may determine $m\le n$ and $c_0,c_1,\ldots,c_{m-1}$ in $\mathcal{X}$ such that $-\pi<c_0<\dots<c_{m-1}<\pi$ and $\mathcal{Y} =\{-\pi,c_0,\dots,c_{m-1},\pi\}$. 
   
   Now let $b_0,a_1,b_1, \ldots$ be given such that, for all $x$ in $\mathcal{X}$, $$F(x):=\frac{b_0}{2}+\sum_{n>0}a_n\sin nx + b_n\cos nx=0.$$

   We make a \textit{second provisional assumption:} the infinite sequence $b_0, a_1, b_1, \ldots$ is \textit{bounded}.
   
   The function $$G(x) := \frac{1}{4}b_0x^2 +\sum_{n>0}\frac{- a_n}{n^2} \sin nx + \frac{-b_n}{n^2} \cos nx$$ then is defined everywhere on $[-\pi,\pi]$ and everywhere continuous, 
    and, according to Riemann's first result, for all $x$ in $\mathcal{X}$, $D^2G(x) =0$. Find $n_0$ such that $-\pi +\frac{1}{2^{n_0-1}} < c_0 $. According to Corollary \ref{cor:cs}, for each $n\ge n_0$, $G$ is linear on $[-\pi+\frac{1}{2^n}, c_0-\frac{1}{2^n}]$. Conclude: $G$ is linear on $(-\pi, c_0)$. For similar reasons, $G$ is linear on each of the intervals $(c_0, c_1)$, $(c_1, c_2), \ldots , (c_{m-1}, \pi)$. 
    
    Find $d_0, e_0, d_1, e_1$ such that, for every $x$ in $(-\pi, c_0)$, $G(x) = d_0x +e_0$ and, for every $x$ in $(c_0,c_1)$, $G(x) = d_1x + e_1$.  According to Riemann's second result,  $D^1G(c_0)=0$, and, therefore: $d_0 = d_1$ and also: $e_0=e_1$. We thus see: $G$ is linear on $(-\pi, c_1)$ and, continuing the argument, we find: $G$ is linear on $(-\pi, \pi)$ and thus also on $[-\pi,\pi]$. Now conclude, as in the proof of Theorem \ref{T:cu}: $b_0=0$ and, for all $n>0$, $a_n = b_n =0$.
    
    \smallskip \textit{ One may do without he first provisional assumption}. As the set $\mathcal{Y}$ is finitely enumerable, one may prove intuitionistically\footnote{One may prove this by induction on the length of the enumeration, using: $\neg \neg(x \;\#_\mathcal{R}\; y \;\vee\; x =_\mathcal{R} y)$ and $(\neg\neg A \;\wedge \;\neg\neg B) \rightarrow \neg\neg(A\;\wedge\;B)$.}: $\neg\neg\bigl(\forall x \in \mathcal{Y}\forall y \in \mathcal{Y}[x=_\mathcal{R} y \;\vee\; x\;\#_\mathcal{R}\;y]\bigr)$. The proof thus far established: $$\forall x \in \mathcal{Y}\forall y \in \mathcal{Y}[x=_\mathcal{R} y \;\vee\; x\;\#_\mathcal{R}\;y]\rightarrow b_0 =0.$$
    Taking two times the contraposition, we obtain:  $$\neg\neg\forall x \in \mathcal{Y}\forall y \in \mathcal{Y}[x=_\mathcal{R} y \;\vee\; x\;\#_\mathcal{R}\;y]\rightarrow \neg\neg(b_0 =0).$$ Therefore: $\neg\neg(b_0=0)$. As is well-known, equality on the reals is a  \textit{stable} relation, that is, one may conclude:  $b_0 =0$.
    
     For similar reasons: for all $n>0$, $a_n = b_n = 0$.
    
     \smallskip \textit{One may do without the second provisional assumption}. We no longer assume that  the infinite sequence $b_0, a_1, b_1, \ldots$ is bounded.
   
   Using the suggestion made by Riemann and  Kronecker, we reason as follows.
   
    Assume: $x \in [-\pi,\pi]$ and $\forall i<n[x\;\#_\mathcal{R}\; x_i]$. Define for each $t$ in $[-\pi, \pi]$, $$ K(t) := F(x+t) + F(x-t)$$ and note: $$K(t) = b_0 + 2\sum_{n>0}(a_n\sin nx + b_n \cos nx)\cos nt,$$
    
    Let $\mathcal{Y}$ be the set of all numbers of one of the forms  $x_i-x$, $x_i+x$, where $i<n$.
   Note: for each $t$ in $[-\pi,\pi]$, if $\forall y \in \mathcal{Y}[ t\;\#_\mathcal{R}\;y]$ then $\forall i<n[x+t\;\#_\mathcal{R}\; x_i \;\wedge\; x-t\;\#_\mathcal{R}\; x_i]$ and $K(t)=0$. Also, as $\forall i<n[x\;\#_\mathcal{R}\; x_i]$,  the sequence $n \mapsto a_n\sin nx + b_n \cos nx$ converges and is bounded. Using the first part of the proof, we conclude: $b_0 = 0$ and, for each $n>0$,  $ a_n\sin nx + b_n \cos nx=0$. This conclusion holds for all $x$  in $[-\pi,\pi]$ such that $\forall i<n[x\;\#_\mathcal{R}\; x_i]$.  We conclude: $b_0 = 0$ and, for each $n>0$, $a_n = b_n = 0$. \end{proof}

   \section{Every  open subset of $[-\pi,\pi]$ that is eventually full guarantees uniqueness}
   
  \subsection{An important Lemma} 
   
   Let $\mathcal{G}$ be a subset of $\mathcal{R}$. $\mathcal{G}$ is \textit{open} if and only if, for each $x$ in $\mathcal{G}$, there exists $n$ such that $(x-\frac{1}{2^n}, x+\frac{1}{2^n})\subseteq \mathcal{G}$.  
   
   Let $a,b$ in $\mathcal{R}$ be given such that $a<b$. Let $\mathcal{G}$ be an open subset of $(a,b)$ and let $H$ be a function from $(a,b)$ to $\mathcal{R}$. We define: $H$ \textit{is locally linear on} $\mathcal{G}$ if and only if, for each $x $ in $\mathcal{G}$ there exists $n$ such that $(x-\frac{1}{2^n}, x+\frac{1}{2^n})\subseteq \mathcal{G}$ and $H$ is linear on $(x-\frac{1}{2^n}, x+\frac{1}{2^n})$.
  \begin{lemma}\label{L:lll} For all $a,b$ in $\mathcal{R}$ such that $ a<b$, for every function $H$ from $(a,b)$ to $\mathcal{R}$, if $H$ is  locally linear on $(a,b)$, then $H$ is  linear on $(a,b)$.
    \end{lemma}
    
    \begin{proof}The proof we give is elementary in the sense that it avoids the use of the Heine-Borel Theorem, or, equivalently, the Fan Theorem.
    
     Assume: $a<x_0<x_1<b$.  Find $m$ in $\mathbb{N}$, $d_0, d_1, e_0, e_1$ in $\mathcal{R}$, such that: $(x_0 -\frac{1}{2^{m}}, x_0 + \frac{1}{2^{m}})\subseteq \mathcal{G}$ and, for all $x$ in  $( x_0 - \frac{1}{2^{m}}, x_0 + \frac{1}{2^{m}})$, $H(x) = d_0x+e_0$, and: $(x_1 -\frac{1}{2^{m}}, x_1 + \frac{1}{2^{m}})\subseteq \mathcal{G}$, and, for all $x$ in  $( x_1 - \frac{1}{2^{m}}, x_1 + \frac{1}{2^{m}})$, $H(x) = d_1x+e_1$. We want to prove: $d_0 = d_1$ and $e_0=e_1$. Assume: $d_0 \;\#_\mathcal{R}\;d_1$. We will obtain a contradiction by the method of \textit{successive bisection}. 
    
    We define an infinite sequence $(a_n, d_{n,0}, e_{n,0},  b_n,  d_{n,1}, e_{n,1})_{n\in \mathbb{N}}$ 
    of sextuples of reals  such that $(a_0, d_{0,0}, e_{0,0},  b_0,  d_{0,1}, e_{0,1})=(x_0, d_0, e_0, x_1, d_1, e_1)$, and, for each $n$, \begin{enumerate}[\upshape 1.]  \item 
     either: $(a_{n+1}, b_{n+1}) = (a_n,\frac{1}{2}(a_n + b_n))$, or: $(a_{n+1}, b_{n+1}) = (\frac{1}{2}(a_n + b_n), b_n)$, and \item there exists $m$ such that: $( a_n - \frac{1}{2^{m}}, a_n + \frac{1}{2^{m}}) \subseteq \mathcal{G}$ and,   for all $x$ in  $( a_n - \frac{1}{2^{m}}, a_n + \frac{1}{2^{m}}) $, $H(x) = d_{n,0}x+e_{n,0}$,  and: $( b_n - \frac{1}{2^{m}}, b_n + \frac{1}{2^{m}})\subseteq \mathcal{G}$ and, for all $x$ in  $( b_n - \frac{1}{2^{m}}, b_n + \frac{1}{2^{m}})$,  $H(x) = d_{n,1}x+e_{n,1}$, and \item  $d_{n,0} \;\#_\mathcal{R}\; d_{n,1}$.\end{enumerate} We do so as follows.  Suppose $n$ is given and $(a_n, d_{n,0}, e_{n,0},  b_n,  d_{n,1},e_{n,1})$  has been defined. Define $c_n := \frac{a_n+b_n}{2}$ and find $d,e,n$ such that  for all $x$ in  $( c - \frac{1}{2^{n}}, c + \frac{1}{2^{n}})$, $H(x) = dx+e$. Note: either $d \;\#_\mathcal{R} \; d_{n,0}$ or $d \;\#_\mathcal{R} \; d_{n,1}$. If you first discover  $d \;\#_\mathcal{R} \; d_{n,0}$, define: $a_{n+1} := a_n$, $d_{n+1,0} :=d_{n,0}$ and $e_{n=1,0} := e_{n,0}$ and $b_{n+1} :=c$, $d_{n+1,1} := d$ and $e_{n+1,1} := e$, and,  if you first discover  $d \;\#_\mathcal{R} \; d_{n,1}$, define: $a_{n+1} := c$, $d_{n+1,0} :=d$ and $e_{n=1,0} := e$ and $b_{n+1} :=b_n$, $d_{n+1,1} := d_{n,1}$ and $e_{n+1,1} := e_{n,1}$. We use the expression \textit{`first discover'} in order to indicate that one may formulate a rule as to which of the two alternatives we should choose in terms of the way the numbers $d, d_{n,0}$ and $d_{n,1}$ are given to us as infinite sequences of rational approximations. One may avoid the use of an Axiom of Countable Choice.
    
    Now find $z$ such that, for each $n$, $a_n \le z \le b_n$. Find $d,e,m$ such that  $( z - \frac{1}{2^{m}}, z + \frac{1}{2^{m}})\subseteq \mathcal{G}$ and, for all $x$ in  $( z - \frac{1}{2^{m}}, z + \frac{1}{2^{m}})$, $H(x) = dx+e$. Find $n_0$ such that $b_{n_0} - a_{n_0} < \frac{1}{2^m}$ and conclude: $d=d_{n_0, 0}=d_{n_0,1}$. Contradiction.
    
    We have to admit: $d_0 = d_1$, and: $e_0 = e_1$, and: $G$ is linear on $(a,b)$.\end{proof}
    
 \subsection{The co-derivative extension of an open set}  Let $\mathcal{G}$ be an open subset of $\mathcal{R}$. We define: $$\mathcal{G}^+:=\{x\in\mathcal{R}|\exists n\in\mathbb{N}\exists y\in\mathcal{R}[|x-y|<\frac{1}{2^n}\;\wedge\;\forall z[(|x-z|< \frac{1}{2^n}\;\wedge\; z\;\#_\mathcal{R}\; y) \rightarrow z \in \mathcal{G}]]\}.$$  $\mathcal{G}^+$ is called: the \textit{(first) co-derivative extension of $\mathcal{G}$}. 
 
 Note: $x\in\mathcal{G}^+$ if and only if all members of  some neighbourhood of $x$ are in  $\mathcal{G}$ \textit{with one possible and well-known exception}. 
   
   Note: $\mathcal{G} \subseteq \mathcal{G}^+$ and: $\mathcal{G}^+$ is an open subset of $\mathcal{R}$.

  \smallskip Note: if $\mathcal{G}^+$ is \textit{inhabited}, that is, one may effectively find an element of $\mathcal{G}^+$, then so is $\mathcal{G}$.
  
  \smallskip
   The next Lemma explains the use Cantor is making of Riemann's second result. 
   \begin{lemma}\label{L:cv} For all $a,b$ in $\mathcal{R}$ such that $a<b$, for  every open subset $\mathcal{G}$ of $(a,b)$,  for every function $H$ from $(a,b)$ to $\mathcal{R}$, if, for all $x$ in $(a,b)$, $D^1H(x) =0$, and $H$ is locally linear on $\mathcal{G}$, then $H$ is locally linear on the co-derivative extension $\mathcal{G}^+$ of $\mathcal{G}$.\end{lemma}
\begin{proof} Assume: $x \in \mathcal{G}^+$. Find $n,y$ such that $|x-y|<\frac{1}{2^n}$ and $\forall z[(|x-z|< \frac{1}{2^n}\;\wedge\; z\;\#_\mathcal{R}\; y) \rightarrow z \in \mathcal{G}]$.
 Note: both $ (x-\frac{1}{2^n}, y)$ and $ (y, x+\frac{1}{2^n})$ are subsets of $\mathcal{G}$, so $H$ is locally linear on both  $ (x-\frac{1}{2^n}, y)$ and $ (y, x+\frac{1}{2^n})$. Using Lemma \ref{L:lll}, we conclude that $H$ is  linear on both  $ (x-\frac{1}{2^n}, y)$ and $ (y, x+\frac{1}{2^n})$. Find $d_0, e_0, d_1, e_1$ such that: for every $x$ in $ (x-\frac{1}{2^n}, y)$, $H(x) = d_0x +e_0$ and: for every $x$ in $ (y, x+\frac{1}{2^n})$, $H(x) = d_1x + e_1$. As $D^1H(y)=0$,  $d_0 = d_1$ and also: $e_0=e_1$.  Therefore, $H$ is linear on $(x-\frac{1}{2^n}, x+\frac{1}{2^n})$.
 
 We thus see: $H$ is locally linear on $\mathcal{G}^+$.
\end{proof}
\subsection{Repeating the operation of taking the co-derivative extension} Let $\mathcal{G}$ be an open subset of $\mathcal{R}$. We let $Ext_\mathcal{G}^{<\omega}$ be the smallest collection $\mathcal{E}$ of subsets of $\mathcal{R}$ such that \begin{enumerate}[\upshape (i)] \item $\mathcal{G} \in \mathcal{E}$, and \item for every $\mathcal{H}$ in $\mathcal{E}$, also $\mathcal{H}^+ \in \mathcal{E}$.\end{enumerate}

We let $Ext_\mathcal{G}$ be the smallest collection $\mathcal{E}$ of subsets of $\mathcal{R}$ such that \begin{enumerate}[\upshape (i)] \item $\mathcal{G} \in \mathcal{E}$, and \item for every $\mathcal{H}$ in $\mathcal{E}$, also $\mathcal{H}^+ \in \mathcal{E}$, and \item for every infinite sequence $\mathcal{H}_0, \mathcal{H}_1, \mathcal{H}_2, \ldots$ of elements of $\mathcal{E}$, also $\bigcup\limits_{n \in \mathbb{N}} \mathcal{H}_n \in \mathcal{E}$.\end{enumerate} The elements of $Ext_\mathcal{G}$ are called
\textit{the (co-derivative) extensions of $\mathcal{G}$}. 

A definition by \textit{transfinite} or \textit{generalized induction} like the definition of $Ext_\mathcal{G}$ we just gave is acceptable in intuitionistic mathematics, although Brouwer has not always  been  clear on this point. We may use the following \textit{principle of transfinite induction on $Ext_\mathcal{G}$}:

For every collection $\mathcal{E}'$ of subsets of $\mathcal{R}$, if \begin{enumerate}[\upshape (i)] \item $\mathcal{G} \in \mathcal{E}'$, and \item for every $\mathcal{H}$ in $\mathcal{E}'$, also $\mathcal{H}^+ \in \mathcal{E}'$, and \item for every infinite sequence $\mathcal{H}_0, \mathcal{H}_1, \mathcal{H}_2, \ldots$ of elements of $\mathcal{E}'$, also $\bigcup\limits_{n \in \mathbb{N}} \mathcal{H}_n \in \mathcal{E}'$,\end{enumerate} then $\mathcal{E} \subseteq \mathcal{E}'$.

\smallskip
 The elements of $Ext_\mathcal{G}^{<\omega}$ are called
\textit{the (co-derivative) extensions of $\mathcal{G}$ of finite rank}.

Note: for all $a,b$ in $\mathcal{R}$ such that $a<b$, if $\mathcal{G}\subseteq (a,b)$, then, for every $\mathcal{H}$ in $Ext_\mathcal{G}$, $\mathcal{H} \subseteq (a,b)$.

Also: for every $\mathcal{H}$ in $Ext_\mathcal{G}$, if $\mathcal{H}$ is inhabited, then so is $\mathcal{G}$.

One may prove these facts by  transfinite induction on $Ext_\mathcal{G}$.

  \begin{lemma}\label{L:cv2}For all $a,b$ in $\mathcal{R}$ such that $a<b$, for every open subset $\mathcal{G}$ of $(a,b)$, for every function $H$ from $(a,b)$ to $\mathcal{R}$, if, for all $x$ in $(a,b)$, $D^1H(x) = 0$, and  $H$ is  locally linear on $\mathcal{G}$, then $H$ is locally linear on every element of $Ext_\mathcal{G}$.\end{lemma} 
  \begin{proof} The proof uses Lemma \ref{L:cv} and is by  transfinite induction on $Ext_\mathcal{G}$. \end{proof}
  Let $a,b$ in $\mathcal{R}$ be given such that $a<b$ and let $\mathcal{G}$  be an open subset of $(a,b)$. $\mathcal{G}$ is called \textit{swiftly full in $(a,b)$} if and only if $(a,b) \in Ext_\mathcal{G}^{<\omega}$. $\mathcal{G}$ is called \textit{eventually full in $(a,b)$} if and only if $(a,b) \in Ext_\mathcal{G}$.
 \footnote{ Cantor says: \textit{$[a,b]\setminus \mathcal{G}$ is reducible} where we use the expression: \textit{$\mathcal{G}$ is swiftly/eventually full in $(a,b)$}.} 
 
 Note: if $\mathcal{G}$ is eventually full in $(a,b)$, then $\mathcal{G}$ is inhabited. 
  
  \smallskip
  
  For every open subset $\mathcal{G}$ of $(-\pi,\pi)$, for each $x$ in $(-\pi,\pi)$,  we  let $x+_\pi\mathcal{G}$ be the set of all $t$ in $(-\pi,\pi)$ such that either: $-x+t\in\mathcal{G}$ or: $-x+t+2\pi\in\mathcal{G}$ or: $-x+t-2\pi\in\mathcal{G}$.
  
  We need the following  observation.
  
  \begin{lemma}\label{L:help1}For every open subset $\mathcal{G}$ of $(-\pi,\pi)$, for every $x$ in $(-\pi,\pi)$,
  \begin{enumerate}[\upshape (i)] \item  $x+_\pi\mathcal{G}$ is an open subset of $(-\pi,\pi)$, and \item $Ext_{x+_\pi\mathcal{G}} = \{x+_\pi\mathcal{Y}|\mathcal{Y} \in Ext_\mathcal{G}\}$, and \item  if $\mathcal{G}$ is eventually full in $(-\pi,\pi)$, then $x+_\pi\mathcal{G}$ is eventually full in $(-\pi,\pi)$. \end{enumerate} \end{lemma}
 \begin{proof} The proof is straightforward and left to the reader. \end{proof} 
\begin{lemma}\label{L:prepext} For all open subsets $\mathcal{G}_0, \mathcal{G}_1$ of  $\mathcal{R}$, $({\mathcal{G}_0})^+\cap\mathcal{G}_1\subseteq (\mathcal{G}_0\cap \mathcal{G}_1)^+$.\end{lemma}
   
  \begin{proof} Assume: $x \in (\mathcal{G}_0)^+\cap \mathcal{G}_1$. Find $n, y$ such that $|x-y|<\frac{1}{2^n}$ and $(x-\frac{1}{2^n}, y)\cup(y, x+\frac{1}{2^n})\subseteq \mathcal{G}_0$. Find $m>n$ such that $(x-\frac{1}{2^m}, x +\frac{1}{2^m})\subseteq \mathcal{G}_1
  $ and $y \;\#_\mathcal{R}\; x-\frac{1}{2^m}$ and  $y\;\#_\mathcal{R}\; x+\frac{1}{2^m}$. Then \textit{either:} $y<x -\frac{1}{2^m}$, and: $x\in \mathcal{G}_0\cap\mathcal{G}_1$, \textit{or:} $x-\frac{1}{2^m} < y<x +\frac{1}{2^m}$ and: $x \in \bigl(\mathcal{G}_0\cap\mathcal{G}_1\bigr)^+$, \textit{or:} $y>x +\frac{1}{2^m}$, and: $x\in \mathcal{G}_0\cap\mathcal{G}_1$.  In any case: $x \in \bigl(\mathcal{G}_0 \cap\mathcal{G}_1\bigr)^+$.\end{proof}

 \begin{lemma}\label{L:help2}Let $a,b$ in $\mathcal{R}$ be given such that $a<b$. 
 \begin{enumerate}[\upshape (i)]\item For all open subsets $\mathcal{G}_0, \mathcal{G}_1$ of $(a,b)$, for all $\mathcal{X}$ in $Ext_{\mathcal{G}_0}$, there exists $\mathcal{Y}$ in  $Ext_{\mathcal{G}_0\cap\mathcal{G}_1}$ such that $\mathcal{X} \cap \mathcal{G}_1 \subseteq \mathcal{Y}$. \item For all open subsets $\mathcal{G}_0, \mathcal{G}_1$ of $(a,b)$, if both $\mathcal{G}_0$ and $\mathcal{G}_1$ are eventually full in $(a,b)$, then $\mathcal{G}_0 \cap \mathcal{G}_1$ is eventually full in $(a,b)$. \end{enumerate}\end{lemma}
  \begin{proof} (i) We use  transfinite induction on $Ext_{\mathcal{G}_0}$.
  
   \begin{enumerate}[\upshape 1.] \item $\mathcal{G}_0\cap\mathcal{G}_1\subseteq\mathcal{G}_0\cap\mathcal{G}_1$.
   
   \item Assume: $\mathcal{X} \in Ext_{\mathcal{G}_0}$ and $\mathcal{Y} \in Ext_{\mathcal{G}_0\cap \mathcal{G}_1}$ and $\mathcal{X} \cap \mathcal{G}_1 \subseteq \mathcal{Y}$. Use Lemma \ref{L:prepext} and note: $\mathcal{X}^+ \cap \mathcal{G}_1 \subseteq (\mathcal{X} \cap \mathcal{G}_1)^+\subseteq \mathcal{Y}^+$. \item Let $\mathcal{X}_0, \mathcal{X}_1, \ldots$ and $\mathcal{Y}_0, \mathcal{Y}_1, \ldots$ be infinite sequences of elements of $Ext_{\mathcal{G}_0}$ and $Ext_{\mathcal{G}_0 \cap \mathcal{G}_1}$, respectively, such that, for each $n$, $\mathcal{X}_n \cap \mathcal{G}_1 \subseteq \mathcal{Y}_n$. Note: $\bigcup\limits_{n \in \mathbb{N}} \mathcal{X}_n \in Ext_{\mathcal{G}_0}$ and  $\bigcup\limits_{n \in \mathbb{N}} \mathcal{Y}_n \in Ext_{\mathcal{G}_0\cap\mathcal{G}_1}$ and  $\bigl(\bigcup\limits_{n \in \mathbb{N}} \mathcal{X}_n\bigr)\cap \mathcal{G}_1 \subseteq \bigcup\limits_{n \in \mathbb{N}} \mathcal{Y}_n$.\end{enumerate} 
   
   \smallskip (ii) Assume: $\mathcal{G}_0$ and $\mathcal{G}_1$ are eventually full in $(a,b)$. Then: $(a,b) \in Ext_{\mathcal{G}_0}$. Using (i), find $\mathcal{Y}$ in $Ext_{\mathcal{G}_0\cap\mathcal{G}_1}$ such that $\mathcal{G}_1 \subseteq \mathcal{Y}$. One may prove now, by transfinite induction on  $Ext_{\mathcal{G}_1}$:   for each $\mathcal{X}$ in $Ext_{\mathcal{G}_1}$ there exists $\mathcal{Z}$ in $Ext_{\mathcal{G}_0\cap\mathcal{G}_1}$ such that $\mathcal{X} \subseteq \mathcal {Z}$.  In particular, there exists $\mathcal{Z}$ in $Ext_{\mathcal{G}_0 \cap\mathcal{G}_1}$ such that $(a,b)\subseteq \mathcal{Z}$, and, therefore: $(a,b) \in  Ext_{\mathcal{G}_0 \cap \mathcal{G}_1}$, that is:   $\mathcal{G}_0 \cap \mathcal{G}_1$ is eventually full in $(a,b)$.   \end{proof} \begin{theorem}\label{T:evfullunique}Every  open subset of $(-\pi,\pi)$ that is eventually full in $(\pi,\pi)$ guarantees uniqueness. \end{theorem}
   \begin{proof} Let $\mathcal{G}$ be an open subset of $(-\pi, \pi)$ that is eventually full.
    Let $b_0,a_1,b_1, \ldots$ be given such that, for all $x$ in $\mathcal{G}$, $F(x):=\frac{b_0}{2}+\sum_{n>0}a_n\sin nx + b_n\cos nx=0$.
   
   \smallskip
    \textit{We make a provisional assumption}: the infinite sequence $b_0, a_1, b_1, \ldots$ is bounded.
   
   The function $$G(x) := \frac{1}{4}b_0x^2 +\sum_{n>0}\frac{- a_n}{n^2} \sin nx + \frac{-b_n}{n^2} \cos nx$$ is therefore defined everywhere on $[-\pi,\pi]$ and everywhere continuous, 
    and, according to Riemann's first result, for all $x$ in $\mathcal{G}$, $D^2G(x) =0$.  Using Corollary \ref{cor:cs}, we conclude: $G$ is locally linear on $\mathcal{G}$. Riemann's second result 
  is  that, for all $x$ in $[-\pi,\pi]$, $D^1G(x) =0$. Using Lemma \ref{L:cv2}, we conclude: $G$ is locally linear on every co-derivative extension of $\mathcal{G}$. As $\mathcal{G}$ is eventually full, $(-\pi,\pi)$ is such an extension, and $G$ is locally linear on $(-\pi,\pi)$, and therefore, according to Lemma \ref{L:lll}, $G$ is linear on $[-\pi, \pi]$.
   As in the proof of Theorem \ref{T:cu}, we conclude: $b_0 =0$ and, for all $n>0$, $a_n = b_n = 0$.
  
   \smallskip \textit{One may do without  the  provisional assumption}. We no longer assume that  the infinite sequence $b_0, a_1, b_1, \ldots$ is bounded.
   
   Again using the suggestion made by Riemann and  Kronecker, we reason as follows.
   
    Assume: $x \in \mathcal{G}$. Define for each $t$ in $[-\pi, \pi]$, $$ K(t) := F(x+t) + F(x-t)$$ and note: $$K(t) = b_0 + 2\sum_{n>0}(a_n\sin nx + b_n \cos nx)\cos nt,$$

    Note: for each $t$ in $[-\pi,\pi]$, if both $x+t \in \mathcal{G}$ and $x-t\in \mathcal{G}$ then $K(t)=0$. Note that $-\mathcal{G} :=\{-y|y \in \mathcal{G}\}$ is, like $\mathcal{G}$ itself, an eventually full open subset of $\mathcal{G}$.  The set $(-x+_\pi \mathcal{G}) \cap (x+_\pi (-\mathcal{G}))$ is eventually full, according to Lemmas \ref{L:help1} and \ref{L:help2}, and, for  all $t$ in $(-x+_\pi \mathcal{G}) \cap (x+_\pi(- \mathcal{G}))$, $K(t)=0$. In addition, as $x\in \mathcal{G}$, $F(x)=0$  and the sequence $n \mapsto a_n\sin nx + b_n \cos nx$ converges and is bounded. Using the first part of the proof, we conclude: $b_0 = 0$ and, for each $n>0$,  $ a_n\sin nx + b_n \cos nx=0$. This conclusion holds for all $x$ in  $\mathcal{G}$. As $\mathcal{G}$ is eventually full, $\mathcal{G}$ is inhabited and there exist $c,d$ such that $c<d$ and, for all $n>0$, for all $x$ in $(c,d)$,  $ a_n\sin nx + b_n \cos nx=0$. We conclude:  for all $n>0$, $a_n = b_n = 0$.\end{proof}
    
    Cantor proved, in \cite{cantor72}, (the classical equivalent of the statement) that every \textit{swiftly} full open subset of $(-\pi,\pi)$ guarantees uniqueness.
    
    Note that we gave an \textit{ordinal-free} treatment of the statement that every \textit{eventually} full open subset of $(-\pi,\pi)$ guarantees uniqueness. The classical version of this result,  clearly present in Cantor's mind, appears in print in a paper by H.~Lebesgue, see \cite{lebesgue03}. As we just saw, one does not need ordinals in order to obtain this result. The set $Ext_\mathcal{G}$ had to be introduced by a (generalized) inductive definition. Using countable ordinals does not make it possible to avoid such definitions, as the class of all countable ordinals itself has to be introduced by one.

    \section{Some examples}
 In this Section, we want to show  that there exists a rich variety of open subsets of $(-\pi,\pi)$ that are eventually full. \subsection{Cantor-Bendixson sets}  For all $a,b$ in $\mathcal{R}$ such that $a<b$, we define  a collection  $En_{(a,b)}$ of functions from  $\mathbb{N}$ to $[a,b]$. 
    
    \begin{enumerate}[\upshape (i)]\item For all $a,b$ in $\mathcal{R}$ such that $a<b$, the function $f$ satisfying: $f(0)=a$ and, for all $n$, $f(n+1) = b$, belongs to $En_{(a,b)}$, and,
     
    \item for all $a,b,c$ in $\mathcal{R}$ such that $a<c<b$, for all infinite sequences of reals $a=a_0, a_1, a_2, \ldots$ and $b=b_0,b_1,b_2,\ldots$ such that $\forall n[a_n < a_{n+1}<c<b_{n+1}<b_n]$ and $\forall m \exists n[b_n-a_n<\frac{1}{2^m}]$,  for every infinite sequence $f_0, f_1, f_2, \ldots$ such that, for all $n$, $f_{2n} \in En_{(a_n, a_{n+1})}$ and $f_{2n+1} \in En_{(b_{n+1}, b_{n})}$, the function $f$ satisfying $f(0) = a$, $f(1) = b$ and $f(2) =c$, and, for all $n$, for all $m$, $f\bigl(2^n(2m+1)+2\bigr) = f_n(m)$, belongs to $En_{(a,b)}$.    
     \end{enumerate}
   
   The elements of $En_{(a,b)}$ are called the \textit{Cantor-Bendixson enumerations associated with the open interval $(a,b)$}.

   For every $a,b$ in $\mathcal{R}$ such that $a<b$, for every $f$ in $En_{(a,b)}$, we let $\mathcal{G}_{(a,b)}^f$ be the set of all $x$ in $[a,b]$ such that, for all $n$, $f(n) \;\#_\mathcal{R}\; x$. 
   
   Note that, in general, if $f$ is a function from $\mathbb{N}$ to $(a,b)$, then the set of all $x$ in $(a,b)$ such that, for all $n$, $f(n) \;\#_\mathcal{R}\;x$, is not an open subset of $(a,b)$ but a countable intersection of open subsets of $(a,b)$.

    \begin{theorem} Let $a,b$ in $\mathcal{R}$ be given such that $a<b$.
    
    For every $f$ in $En_{(a,b)}$, the set $\mathcal{G}_{(a,b)}^f$ is an open subset of $(a,b)$.   \end{theorem} 
 \begin{proof}  We use induction.  Assume: $f\in En_{(a,b)}$. Note: either $f(2)=b$ or $f(2)<b$. 
 
    If   $f(2) = b$, then  $\mathcal{G}_{(a,b)}^f = (a,b)$.
    
    Assume $f(2)<b$.  Define $c:=f(0)$. Define, for each $n$ a function $f_n$ from $\mathbb{N}$ to $[a,b]$ such that, for all $m$, $f_n(m) :=f\bigl(2^n(2m+1)+2\bigr)$. Define, for each $n$, $a_n := f_{2n}(0)$ and $b_n:=f_{2n+1}(1)$.  Note that, for all $n$, $f_{2n} \in En_{(a_n, a_{n+1})}$ and $f_{2n+1} \in En_{(b_{n+1}, b_{n})}$, and:
   $\mathcal{G}_{(a,b)}^f= \bigcup\limits_{n \in \mathbb{N}} \mathcal{G}_{(a_n,a_{n+1})}^{f_{2n}}\cup  \mathcal{G}_{(b_{n+1},b_n)}^{f_{2n+1}}$. Assuming that, for each $n$, $\mathcal{G}_{(a_n,a_{n+1})}^{f_{2n}}$ and $\mathcal{G}_{(b_{n+1},b_n)}^{f_{2n+1}}$ are open subsets of $(a_n, a_{n+1})$ and $(b_{n+1}, b_n)$, respectively, we conclude that $\mathcal{G}_{(a,b)}^f$ is an open subset of $(a,b)$. 
   \end{proof}

   \begin{lemma}\label{L:ext} Let $a,b,c,d$ in $\mathcal{R}$ be given such that $a<c<d<b$. Let $\mathcal{G}$ be an open subset of $(a,b)$. Then $Ext_{\mathcal{G}\cap(c,d)} = \{\mathcal{X} \cap (c,d)|\mathcal{X} \in Ext_\mathcal{G}\}$. \end{lemma}
   \begin{proof} We use induction. Clearly, $(c,d) = (a,b)\cap(c,d)$.
   Note, using Lemma \ref{L:prepext}, for all $\mathcal{X}$ in $Ext_{\mathcal{G}\cap(c,d)}$, for all $\mathcal{Y}$ in $Ext_\mathcal{G}$, if $\mathcal{X} = \mathcal{Y} \cap(c,d)$, then $\mathcal{X}^+ = \mathcal{Y}^+ \cap(c,d)$. Finally, if $\mathcal{X}_0, \mathcal{X}_1, \ldots$ and $\mathcal{Y}_0, \mathcal{Y}_1, \ldots$ are infinite sequences of elements of $Ext_{\mathcal{G}\cap(c,d)}$ and $Ext_\mathcal{G}$, respectively, such that, for each $n$, $\mathcal{X}_n = \mathcal{Y}_n \cap(c,d)$, then $\bigcup\limits_{n \in \mathbb{N}} \mathcal{X}_n =(\bigcup\limits_{n \in \mathbb{N}} \mathcal{Y}_n)\cap(c,d)$. \end{proof}
   \begin{theorem}\label{T:cbevfull} Let $a,b$ in $\mathcal{R}$ be given such that $a<b$.
    
     For every $f$ in $En_{(a,b)}$, the set $\mathcal{G}_{(a,b)}^f$ is eventually full in $(a,b)$.  \end{theorem} 
     
     \begin{proof} We use induction.  Assume: $f\in En_{(a,b)}$. Note: either $f(2)=b$ or $f(2)<b$. 
 
    If   $f(2) = b$, then  $\mathcal{G}_{(a,b)}^f = (a,b)$.
    
    Assume $f(2)<b$.  Define $c:=f(0)$. Define, for each $n$ a function $f_n$ from $\mathbb{N}$ to $[a,b]$ such that, for all $m$, $f_n(m) :=f\bigl(2^n(2m+1)+2\bigr)$. Define, for each $n$, $a_n := f_{2n}(0)$ and $b_n:=f_{2n+1}(1)$.  Note that, for all $n$, $f_{2n} \in En_{(a_n, a_{n+1})}$ and $f_{2n+1} \in En_{(b_{n+1}, b_{n})}$, and:
   $\mathcal{G}_{(a,b)}^f= \bigcup\limits_{n \in \mathbb{N}} \mathcal{G}_{(a_n,a_{n+1})}^{f_{2n}}\cup  \mathcal{G}_{(b_{n+1},b_n)}^{f_{2n+1}}$. Assuming that, for each $n$, $\mathcal{G}_{(a_n,a_{n+1})}^{f_{2n}}$ is eventually full in $(a_n, a_{n+1})$ and $\mathcal{G}_{(b_{n+1},b_n)}^{f_{2n+1}}$ is eventually full in $(b_{n+1}, b_n)$, respectively, we determine, using Lemma \ref{L:ext}, an infinite sequence $\mathcal{Y}_0, \mathcal{Y}_1, \ldots$ of elements of $Ext_\mathcal{G}$ such that, for each $n$, $\mathcal{Y}_{2n} \cap (a_n, a_{n+1}) = (a_n, a_{n+1})$ and $\mathcal{Y}_{2n+1} \cap (b_{n+1}, b_{n}) = (b_{n+1}, b_{n})$. Define $\mathcal{Y} = \bigcup\limits \mathcal{Y}_n$ and note: $\bigcup\limits_{n \in \mathbb{N}} (a_n, a_{n+1})\cup(b_{n+1}, b_n) \subseteq \mathcal{Y}$. Conclude: $\{x \in (a,b)|x\;\#_\mathcal{R} \;c\} \subseteq \mathcal{Y}^+$ and: $\mathcal{Y}^{++} = (a,b)$. Conclude: $\mathcal{G}$ is eventually full in $(a,b)$. 
 \end{proof}  
    We let $\mathcal{CBO}_{(a,b)}$ be the collection of all sets $\mathcal{G}_{(a,b)}^f$, where $f \in En_{(a,b)}$. 
    
    These sets are called the \textit{Cantor-Bendixson-open subsets of $(a,b)$}. 
    
   \begin{corollary} Every Cantor-Bendixson-open subset of $(-\pi,\pi)$ guarantees uniqueness. \end{corollary}
   
   \begin{proof} Use Theorems \ref{T:evfullunique} and \ref{T:cbevfull}. \end{proof}
    Let $a,b$ in $\mathcal{R}$ be given such that $a<b$.  For every $f$ in $En_{(a,b)}$, we define:
    
     $$\mathcal{F}_{(a,b)}^f:=\{x\in [a,b]|x \notin \mathcal{G}_{(a,b)}^f\}.$$  A set of this form is called a \textit{Cantor-Bendixson-closed subset of $(a,b)$}.
     
   \subsubsection{}\label{SSS:located}  A subset $\mathcal{X}$ of $\mathcal{R}$ is called a \textit{located} subset of $\mathcal{R}$  if and only if, for each $x$ in $\mathcal{R}$, one may find $s$ in $\mathcal{R}$ such that (i) for all $y$ in $\mathcal{X}$, $|y-x| \ge s$, and (ii) for every $\varepsilon >0$  there exists $y$ in $\mathcal{X}$ such that $|y-x| < s+\varepsilon$. This number $s$, if it exists,  is the \textit{infimum} or \textit{greatest lower bound} of the set $\{|y-x||y\in [a,b]\}$, and is called the \textit{distance from $x$ to $\mathcal{X}$} notation: $d(x, \mathcal{X})$. 
     
     \begin{theorem} For all $a,b$ in $\mathcal{R}$ such that $a<b$, every  Cantor-Bendixson-closed subset of $(a,b)$  is a located subset of $\mathcal{R}$. \end{theorem}
     \begin{proof} The proof is by induction.    Let $f$ in $En_{(a,b)}$ be given. Note: either $f(2)=b$ or $f(2) <b$. 
     
     If $f(2) =b$, then $\mathcal{F}_{(a,b)}= \{a,b\}$ is clearly located.
     
      Now assume $f(2)<b$. Define $c:=f(2)$.  Define, for each $n$, a function $f_n$ from $\mathbb{N}$ to $[a,b]$ such that, for all $m$, $f_n(m) := f\bigl(2^n(2m+1)\bigr)$. Define,  for each $n$, $a_n :=f_{2n}(0)$ and $b_n:= f_{2n+1}(1)$.   Note that, for all $n$, $f_{2n} \in En_{(a_{n}, a_{n+1})}$ and $f_{2n+1} \in En_{(b_{n+1}, b_{n})}$. 
      
      Assume: $x\in[a,b]$ and $x \;\#_\mathcal{R}\; c$.
       If $x>c$, find $n$ such that $b_n<x$ and note: $d(x, \mathcal{F}_{(a,b)}^f) = \inf(x-b_n, \inf_{i\le n}d(x, \mathcal{F}_{(b_{i+1},b_{i})}^{f_{2i+1}})$. 
       If $x<c$, find $n$ such that $x<a_n$ and note: $d(x, \mathcal{F}_{(a,b)}^f) = \inf(a_n-x, \inf_{i\le n}d(x, \mathcal{F}_{(a_{i},a_{i+1})}^{f_{2i}})$.
       
       Let $x$ in $[a,b]$ be given. Find $\alpha$ in $\mathcal{C}$ such that, for each $n$, if $\alpha(n) =0$, then $|x-c| < \frac{1}{2^n}$, and, if $\alpha(n) = 1$ then $|x-c|>\frac{1}{2^{n+1}}$. Define, for each $n$, if $\forall i\le n[\alpha(i) =0]$, then $x_n = c +\frac{1}{2^n}$, and, if $\alpha(n) =1$, then $x_n=x$. Note: $d(x, \mathcal{F}_{(a,b)}^f)=\lim_{n\rightarrow \infty}d(x_n, \mathcal{F}_{(a,b)}^f)$.
        \end{proof}
     
     \subsubsection{Ordering  $En_{(-\pi,\pi)}$}\label{SSS:ordercb} One may define relations $\prec_{CB}$ and $\preceq_{CB}$ on the collection of the Cantor-Bendixson-enumerations associated with  $(-\pi,\pi)$, as follows.
  
  \smallskip
  $f \prec_{CB} g$ if and only if there exists a co-derivative extension  $\mathcal{H}$ of the set  $(-\frac{1}{2}\pi+\frac{1}{2}\mathcal{G}_{(-\pi,\pi)}^f) \cup (\frac{1}{2}\pi+\frac{1}{2}\mathcal{G}_{(-\pi,\pi)}^g)$ such that $ (-\pi,0)\subseteq \mathcal{H}$ and \textit{not:} $(0,\pi)\subseteq \mathcal{H}$, and 
  
    \smallskip
   $f \preceq_{CB} g$ if and only if for all co-derivative extensions  $\mathcal{H}$ of the set  $(-\frac{1}{2}\pi+\frac{1}{2}\mathcal{G}_{(-\pi,\pi)}^f) \cup (\frac{1}{2}\pi+\frac{1}{2}\mathcal{G}_{(-\pi,\pi)}^g)$,  if $ (0,\pi)\subseteq \mathcal{H}$, then   $(-\pi,0)\subseteq \mathcal{H}$.
   
   \smallskip $f \prec_{CB} g$ means roughly: \textit{the Cantor-Bendixson-rank of $\mathcal{F}_{(-\pi,\pi)}^f$ is strictly lower than 
  the Cantor-Bendixson-rank of $\mathcal{F}_{(-\pi,\pi)}^g$},  and 
  
  $f \preceq_{CB} g$ means roughly: \textit{the Cantor-Bendixson-rank of $\mathcal{F}_{(-\pi,\pi)}^f$ is not higher than 
  the Cantor-Bendixson-rank of $\mathcal{F}_{(-\pi,\pi)}^g$}.
  
  \smallskip
  One may study these relations and prove important facts, like:
  
  \begin{quote} For every infinite sequence $\mathcal{G}_0$, $\mathcal{G}_1, \dots$ of elements of $\mathcal{CBO}_{(-\pi,\pi)}$ there exists $\mathcal{H}$ in 
  $\mathcal{CBO}_{(-\pi,\pi)}$ such that, for all $n$, $\mathcal{G}_n \prec_{CB} \mathcal{H}$.
  \end{quote}
   \subsection{Perhaps} For every function $f$  from $\mathbb{N}$ to $\mathcal{R}$, we define: $Ran(f) := \{f(n)|n \in \mathbb{N}\}.$ The set $Ran(f)$ is called \textit{the subset of $\mathcal{R}$ enumerated by $f$}.

     \smallskip For each subset $\mathcal{X}$ of $\mathcal{R}$ we define: $\overline{\mathcal{X}}:=\{x \in \mathcal{R}|\forall n \exists y \in \mathcal{X}[|y - x|<\frac{1}{2^n}]\}.$ $\overline{ \mathcal{X}}$ is called the \textit{closure} of the set $\mathcal{X}$.  
     \begin{theorem}\label{T:ccbfanlike}  For all $a,b$ in $\mathcal{R}$  such that $a<b$,  for every $f$ in $En_{(a,b)}$, $\mathcal{F}_{(a,b)}^f = \overline{Ran(f)}$. \end{theorem}
     
     \begin{proof} 
      Note: $Ran(f) \subseteq \mathcal{F}^f_{(a,b)}$ and conclude: $\overline{Ran(f)} \subseteq \mathcal{F}^f_{(a,b)}$. For the converse we use induction.  Assume: $f\in En_{(a,b)}$. Note: either $f(2)=b$ or $f(2)<b$. 
 
    If   $f(2) = b$, then $Ran(f) =\{a,b\}$ and  $\mathcal{F}_{(a,b)}^f = \{a,b\}=\overline{\{a,b\}}$.
    
    Assume $f(2)<b$.  Define $c:=f(0)$. Define, for each $n$ a function $f_n$ from $\mathbb{N}$ to $[a,b]$ such that, for all $m$, $f_n(m) :=f\bigl(2^n(2m+1)+2\bigr)$. Define, for each $n$, $a_n := f_{2n}(0)$ and $b_n:=f_{2n+1}(1)$.  Assume: $x \in \mathcal{F}^f_{(a,b)}$.  Let $k$ be given and note: either: (i)   $|x-c| < \frac{1}{2^k}$ or: (ii)  $x>c$ or (iii) $x<c$. In  case (i): $|x-f(2)| < \frac{1}{2^k}$. In case (ii),  find $k,n$ such that either (ii)a: $|b_n - x|<\frac{1}{2^k}$ or (ii)b: $b_{n+1}<x<b_{n}$. In case (ii)a: $|f(2^{2n+1}\cdot 3 +2) - x|<\frac{1}{2^k}$. In case (ii)b, use the induction hypothesis in order to find $j$ such that $|f\bigl(2^{2n+1}\dot(2j+1)+2\bigr) - x|<\frac{1}{2^k}$.  Case (iii) is handled like case (ii). In all cases: $\exists i[|f(i) -x|<\frac{1}{2^k}]$. We thus see: $\mathcal{F}^f_{(a,b)}\subseteq \overline{Ran(f)}$. \end{proof}

     For every subset $\mathcal{X}$ of $\mathcal{R}$ the class $Perh_\mathcal{X}$  is the least class $\mathcal{E}$ of subsets of $\mathcal{R}$ such that \begin{enumerate}[\upshape (i)] \item $\mathcal{X} \in \mathcal{E}$, and, \item for each $\mathcal{Y}$ in $\mathcal{E}$, also $Perhaps(\mathcal{X}, \mathcal{Y}):= \{y \in \mathcal{R}|\exists x \in \mathcal{X}[x \;\#_\mathcal{R}\; y \rightarrow y \in \mathcal{Y}]\}$ belongs to $\mathcal{E}$, and, \item for every infinite sequence $\mathcal{Y}_0, \mathcal{Y}_1, \mathcal{Y}_2, \ldots$ of elements of $\mathcal{E}$, also $\bigcup\limits_{n \in \mathbb{N}} \mathcal{Y}_n \in \mathcal{E}$. \end{enumerate}
    The elements of $Perh_\mathcal{X}$ are called the \textit{perhapsive extensions of $\mathcal{X}$}.
    
     Note: for all $a,b$ in $\mathcal{R}$ such that $a<b$, if $\mathcal{X}\subseteq [a,b]$, then, for every $\mathcal{Y}$ in $Perh_\mathcal{X}$, $\mathcal{X}\subseteq\mathcal{Y} \subseteq \overline{\mathcal{X}}\subseteq [a,b]$. The perhapsive extensions of $\mathcal{X}$ make us see how large, intuitionistically, the distance may be between the set $\mathcal{X}$ and its closure $\overline{\mathcal{X}}$, even if the classical mathematician is saying: $\mathcal{X}$ \textit{is} closed and coincides with $\overline{\mathcal{X}}$.
     
     The expression \textit{perhapsive} is used for the reason that the sentences  \textit{John is smart, John is perhaps smart, John is perhaps, perhaps smart} seem to decrease in affirmative power. The same is true for the statements $x \in \mathcal{X}, x \in Perh(\mathcal{X}, \mathcal{X}), x\in Perh(\mathcal{X}, Perh(\mathcal{X}, \mathcal{X}))$. 
     
      \begin{lemma}\label{L:prepextperhaps} For all $a,b$ in $\mathcal{R}$ such that $a<b$,  for all subsets $\mathcal{Z}, \mathcal{Y}$ of $\mathcal{R}$, if  $\mathcal{Z}\subseteq\mathcal{Y }$, then   $Perhaps(\mathcal{Z}\cap[a,b],\mathcal{Y}\cap [a,b]) = Perhaps(\mathcal{Z}, \mathcal{Y})\cap[a,b]$.\end{lemma}
      \begin{proof} The proof is left to the reader. \end{proof}
     \begin{lemma}\label{L:extperhaps} Let $a,b,c,d$ in $\mathcal{R}$ be given such that $a<c<d<b$. Let $\mathcal{Z}$ be a  subset of $[a,b]$. Then $Perh_{\mathcal{Z}\cap[c,d]} = \{\mathcal{X} \cap [c,d]|\mathcal{X} \in Perh_\mathcal{Z}\}$. \end{lemma}
   \begin{proof} We use induction. Clearly, $\mathcal{Z}\cap[c,d] = \mathcal{Z}\cap[c,d]$.
   Also, according to Lemma \ref{L:prepextperhaps}, for all $\mathcal{X}$ in $Perh_{\mathcal{Z}\cap[c,d]}$, for all $\mathcal{Y}$ in $Perh_\mathcal{Z}$, if $\mathcal{X} = \mathcal{Y} \cap[c,d]$, then $Perhaps(\mathcal{Z}\cap[c,d], \mathcal{X}) = Perhaps(\mathcal{Z}, \mathcal{Y})\cap[c,d]$. Finally, if $\mathcal{X}_0, \mathcal{X}_1, \ldots$ and $\mathcal{Y}_0, \mathcal{Y}_1, \ldots$ are infinite sequences of elements of $Perh_{\mathcal{Z}\cap[c,d]}$ and $Perh_\mathcal{Z}$, respectively, such that, for each $n$, $\mathcal{X}_n = \mathcal{Y}_n \cap[c,d]$, then $\bigcup\limits_{n \in \mathbb{N}} \mathcal{X}_n =(\bigcup\limits_{n \in \mathbb{N}} \mathcal{Y}_n)\cap[c,d]$. \end{proof}
     
     \smallskip
     A subset $\mathcal{X}$ of $\mathcal{R}$ is \textit{eventually closed} if and only if $\overline{\mathcal{X}} \in Perh_\mathcal{X}$.
      \begin{theorem}\label{T:cbevclosed} Let $a,b$ in $\mathcal{R}$ be given such that $a<b$.
    
     For every $f$ in $En_{(a,b)}$, the set $Ran(f)$ is eventually closed.  \end{theorem} 
     
     \begin{proof} We use induction. Let $f$ in $En_{(a,b)}$ be given. Note: either $f(2) =b$ or $f(2)<b$. If $f(2) =b$, then $Ran(f) = \{a,b\}=\overline{Ran(f)}$. Assume $f(2)<b$. Define, for each $n$, a function $f_n$ from $\mathbb{N}$ to $[a,b]$ such that, for each $m$, $f_n(m) :=f\bigl(2^n(2m+1)\bigr)$. Define $c:= f(0)$ and, for each $n$, $a_n :=f(2^{2n}+2)$ and $b_n := f(2^{2n+1}\cdot 3+2)$. Note: for each $n$, $f_{2n} \in En_{(a_n, a_{n+1})}$ and $f_{2n+1} \in En_{(b_{n+1}, b_{n})}$, and 
   $Ran(f)=\{c\}\cup \bigcup\limits_{n \in \mathbb{N}} Ran(f_{n})$. Assuming that, for each $n$, $Ran(f_n)$ is eventually closed, we determine, using Lemma \ref{L:extperhaps}, an infinite sequence $\mathcal{Y}_0, \mathcal{Y}_1, \ldots$ of elements of $Perh_\mathcal{G}$ such that, for each $n$, $\mathcal{Y}_{2n}\cap[a_n, a_{n+1}] =\overline{Ran(f_{2n})}$ and  $\mathcal{Y}_{2n+1}\cap[b_{n+1}, b_{n}] =\overline{Ran(f_{2n+1})}$. Define $\mathcal{Y} = \bigcup\limits_{n \in \mathbb{N}}\mathcal{Y}_n$ and note: $\mathcal{Y} \in Perh_\mathcal{G}$. 
   
   Assume $x \;\#_\mathcal{R}\; c$ and $x \in \overline{Ran(f)}$. First assume $x>c$. Find $n$ such that $b_{n+2} < c \le b_n$. Note: if $x \;\#_\mathcal{R}\; b_{n+1}$, then either $x<b_{n+1}$ and, therefore, $x\in \overline{Ran(f_{2n+3})}\subseteq \mathcal{Y}_{2n+3} \subseteq \mathcal{Y}$, or  $x>b_{n+1}$ and, therefore, $x\in \overline{Ran(f_{2n+1})}\subseteq \mathcal{Y}_{2n+1} \subseteq \mathcal{Y}$. Conclude: $x \in \mathcal{Y}^+$. Now assume: $x<c$, and conclude, by a similar argument, $x \in \mathcal{Y}^+$. 
   
   We thus see: $\forall x \in \overline{Ran(f)}[x\;\#_\mathcal{R}\;c \rightarrow x \in \mathcal{Y}^+]$. Conclude: $\overline{ Ran(f)}\subseteq \mathcal{Y}^{++} $. Conclude: $Ran(f)$ is eventually closed.
 \end{proof}\subsubsection{Ordering  $En_{(-\pi,\pi)}$ again}\label{SSS:orderperhaps} One may define relations $\prec_{Perh}$ and $\preceq_{Perh}$ on the collection of the Cantor-Bendixson-enumerations  associated with $(-\pi,\pi)$, as follows.
  
  \smallskip
  $f \prec_{Perh}g$ if and only if there exists a perhapsive extension  $\mathcal{H}$ of the set  $\bigl(-\frac{1}{2}\pi+\frac{1}{2}Ran(f)\bigr) \cup \bigl(\frac{1}{2}\pi+\frac{1}{2}Ran(g)\bigr)$ such that $\bigl(-\frac{1}{2}\pi+\frac{1}{2}\overline{Ran(f)}\bigr) \subseteq \mathcal{H}$ and \textit{not:} $ (\frac{1}{2}\pi+\frac{1}{2}\overline{Ran(g)}\bigr) \subseteq \mathcal{H}$, and 
  
    \smallskip
   $f \preceq_{Perh} g$ if and only if for all perhapsive extensions  $\mathcal{H}$ of the set  $\bigl(-\frac{1}{2}\pi+\frac{1}{2}Ran(f)\bigr) \cup \bigl(\frac{1}{2}\pi+\frac{1}{2}Ran(g)\bigr)$, if  $\bigl(-\frac{1}{2}\pi+\frac{1}{2}\overline{Ran(g)}\bigr) \subseteq \mathcal{H}$, then  $  \bigl(\frac{1}{2}\pi+\frac{1}{2}\overline{Ran(f)}\bigr) \subseteq \mathcal{H}$.
   
   \smallskip One might say: $f \prec_{Perh} g$ means: \textit{the  perhapsive rank of $Ran(f)$ is strictly lower than 
  the perhapsive rank of $Ran(g) $},  and 
  
  $f \preceq_{Perh} g$ means: \textit{the perhapsive rank of $Ran(f)$ is not higher than 
  the perhapsive rank of $Ran(g)$}.
  
  \smallskip

  There are close connections between the relations $\prec_{Perh}$ and $\preceq_{Perh}$ and the relations $\prec_{CB}$, $\preceq_{CB}$, introduced in Subsubsection \ref{SSS:ordercb}, see \cite{veldman05}.
  
  For a classical mathematician, this is  (perhaps) embarrassing, as, in his world, for all $f$ in $En_{(-\pi,\pi)}$, for every perhapsive extension $\mathcal{Y}$ of $Ran(f)$, $Ran(f) = \mathcal{Y} =\overline{Ran(f)}$. 
    \section{Every  open subset $\mathcal{G}$ of $[-\pi,\pi]$ such that $[-\pi,\pi]\setminus \mathcal{G}$ is located and almost-enumerable guarantees  uniqueness} \subsection{The complement of an  eventually full open subset of  $(-\pi,\pi)$ is almost-enumerable}  
  
  \smallskip
  
 \subsubsection{} Let $\mathcal{X}$ be a subset of $\mathcal{R}$ and let $f$ be a function from $\mathbb{N}$ to $\mathcal{R}$. 
    
    \textit{$f$ is an enumeration of $X$} or \textit{$f$ enumerates $X$} if and only if $\forall x \in \mathcal{X}\exists n[x= f(n)]$.

     \textit{$f$ is an almost-enumeration of $X$} or \textit{$f$ almost-enumerates $X$} if and only if $\forall x \in \mathcal{X}\forall \gamma \in \mathcal{N}\exists n[\exists n[|f(n) - x|<\frac{1}{2^{\gamma(n)}}]$.

    In order to understand the second definition one should think of $\gamma$ as \textit{possible evidence} for showing: $\forall n[x \;\#_\mathcal{R} f(n)]$: one is hoping: $\forall n[|f(n) - x|\ge\frac{1}{2^{\gamma(n)}}]$.
    $f$ is an almost-enumeration if every possible evidence for the statement  $\forall n[x \;\#_\mathcal{R} f(n)]$ fails in a constructive way.
    
    \smallskip
   A subset $\mathcal{X}$ of $\mathcal{R}$   is \textit{enumerable} or: \textit{almost-enumerable}, respectively, if and only if there exists an enumeration of $\mathcal{X}$, or:   an almost-enumeration  of $\mathcal{X}$, respectively.
    
    Theorem \ref{T:openevfullalmostenumerable}(ii) is a substitute for the classical theorem that every closed and reducible subset of $[-\pi,\pi]$ is at most countable. 
   \begin{theorem}\label{T:openevfullalmostenumerable} Let $\mathcal{G}$ be an open subset of $(-\pi,\pi)$. \begin{enumerate}[\upshape (i)] \item For each $\mathcal{X}$ in $Ext_\mathcal{G}$, for all $a,b$ such that $-\pi\le a<b\le\pi$, if $[a,b] \subseteq \mathcal{X}$, then $[a,b]\setminus \mathcal{G}$ is almost-enumerable, and, \item if $\mathcal{G}$ is eventually full, then the set $[-\pi,\pi]\setminus \mathcal{G}$ is almost-enumerable. 
   \end{enumerate}\end{theorem}  \begin{proof} (i) We use induction on (the definition of) $Ext_\mathcal{G}$. \begin{enumerate}[\upshape 1.]\item Note: if $[a,b] \subseteq \mathcal{G}$, then   $[a,b]\setminus\mathcal{G} = \emptyset$ is almost-enumerable. 
   
 \item  Assume: $\mathcal{X} \in Ext_\mathcal{G}$ and  for all $a,b$ such that $-\pi\le a<b\le\pi$, if $[a,b] \subseteq \mathcal{X}$, then $[a,b]\setminus \mathcal{G}$ is almost-enumerable. 
 
\noindent Assume $-\pi\le a<b\le\pi$ and $[a,b] \subseteq \mathcal{X}^+$.  Then: $$\forall x \in [a,b]\exists c\exists d \exists y[ c<x<d \;\wedge\;c<y<d \;\wedge\; (c,y)\subseteq \mathcal{X}\;\wedge\;(y,d)\subseteq \mathcal{X}].$$ Using the Heine-Borel Theorem we find $n$ in $\mathbb{N}$, and, for each $i<n$, $c_i, d_i,y_i$ such that $$\forall x \in [a,b]\exists i<n[ c_i<x<d_i \;\wedge\;c_i<y_i<d_i \;\wedge\; (c_i,y_i)\subseteq \mathcal{X}\;\wedge\;(y_i,d_i)\subseteq \mathcal{X}].$$ Find $m_0>0$ such that, for each $i<n$, $y_i -c_i>\frac{1}{2^{m_0-1}}$ and $d_i -y_i>\frac{1}{2^{m_0-1}}$. Find, for each $i<n$, for each $m$,  an almost-enumeration $f_{i,m}$ of $[c_i +\frac{1}{2^{m_0 +m}}, y_i -\frac{1}{2^{m_0 +m}}]\setminus \mathcal{G}$ and an almost-enumeration $g_{i,m}$ of $[y_i +\frac{1}{2^{m_0 +m}}, d_i -\frac{1}{2^{m_0 +m}}]\setminus \mathcal{G}$. Define $f$ from $\mathbb{N}$ to $\mathcal{R}$ such that,  for each $i<n$, $f(i) = y_i$ and, for each $m$, for each $i<n$, for each $k$, $f\bigl(n+2^{m\cdot 2n + 2i}(2k+1)\bigr) := f_{i,m}(k)$ and $f\bigl(n+2^{m\cdot 2n + 2i+1}(2k+1)\bigr) = g_{i,m}(k)$. We now prove that $f$ is an almost-enumeration of $[a,b]\setminus \mathcal{G}$. Assume: $x \in [a,b]\setminus \mathcal{G}$ and $\gamma \in \mathcal{N}$.   Find $i<n$ such that  $x \in (c_i, d_i)$.  Note: either (i) $|x - y_i| =|x-f(i)|< \frac{1}{2^{\gamma(i)}}$, or (ii) $x\in (c_i, y_i)$, or: (iii) $x \in (y-i,d_i)$. In case (i), we are sure that $\exists n[|x-f(n)|<\frac{1}{2^{\gamma(n)}}]$. In case (ii), find $m$ such that $x \in [c_i +\frac{1}{2^{m_0 +m}}, y_i -\frac{1}{2^{m_0 +m}}]\setminus \mathcal{G}$. Define $\delta$ in $\mathcal{N}$ such that, for each $k$, $\delta(k) = \gamma\bigl(n+2^{m\cdot 2n + 2i}(2k+1)\bigr)$, and find $k$ such that $|x-f_{i,m}(k)|<\frac{1}{2^{\delta(k)}}$. Define $l:=n+2^{m\cdot 2n + 2i}(2k+1)$ and conclude: $|x-f(l)|<\frac{1}{2^{\gamma(l)}}$. Again, we may conclude: $\exists n[|x-f(n)|<\frac{1}{2^{\gamma(n)}}]$.  In case (iii), we reason similarly. We thus see that $\mathcal{X}^+$ has the required property.
   
   \item Assume: $\mathcal{X}_0, \mathcal{X}_1, \mathcal{X}_2, \ldots$ is an infinite sequence of elements of $  Ext_\mathcal{G}$ and, for all $n$,   for all $a,b$ such that $-\pi\le a<b\le\pi$, if $[a,b] \subseteq \mathcal{X}_n$, then $[a,b]\setminus \mathcal{G}$ is almost-enumerable. Assume: $-\pi\le a<b\le\pi$ and $[a,b]\subseteq \bigcup\limits_{n \in \mathbb{N}}\mathcal{X}_n$. Then $$\forall x \in [a,b]\exists c\exists d \exists n[ c<x<d \;\wedge\; (c,d) \subseteq \mathcal{X}_n].$$Applying the Heine-Borel Theorem, we find $n$ in $\mathbb{N}$ and, for each $i<n$, $c_i, d_i, m_i$ such that $$\forall x \in [a,b] \exists i<n[c_i<x<d_i]\;\wedge\;\forall i<n[[c_i,d_i]\subseteq \mathcal{X}_{m_i}]. $$ 
    
    Find, for each $i<n$, an almost-enumeration $f_i$ of $[c_i,d_i]\setminus \mathcal{G}$. Let $f$ be a function from $\mathbb{N}$ to $\mathcal{R}$ such that, for each $i<n$, for each $m$, $f(m\cdot n +i) = f_i(m)$.  One verifies easily that $f$ is an almost-enumeration of $[a,b]\setminus \mathcal{G}$. We thus see that $\bigcup\limits_{n \in \mathbb{N}}\mathcal{X}_n$ has the required property.\end{enumerate}
    
    \smallskip
      (ii) This is an easy consequence of (i).\end{proof}
 \subsection{A converse result} Theorem \ref{T:ccu} will be  a converse to Theorem    \ref{T:openevfullalmostenumerable}(ii) and a substitute for the classical result that every countable and closed subset of $[-\pi,\pi]$ is reducible. The argument the classical mathematician  uses for this statement is that the kernel of a countable closed set must be empty as a non-empty kernel gives an uncountable set.  We keep far from  thoughts about cardinality.
  
  \subsubsection{}\label{SSS:hblocated} Located subsets of $\mathcal{R}$ entered this paper in Subsubsection \ref{SSS:located}.

  We shall use the following extension of the \textit{Heine-Borel Theorem}: \begin{quote}{\it  Let $a,b$ in $\mathcal{R}$ be given such that $a<b$. Let $\mathcal{G}$ be an open subset of $(a,b)$ such that $[a,b]\setminus \mathcal{G}$ is located. Let $\mathcal{B}$ be a subset of $\mathcal{R}^2$ such that   $\forall x \in [a,b]\setminus \mathcal{G}\exists ( c,d) \in \mathcal{B}[c<x<d]$.   Then  there exist $n$ in $\mathbb{N}$, $( c_0, d_0), ( c_1, d_1), \ldots, ( c_{n}, d_{n})  $ in $\mathcal{B}$ such that $\forall x \in [a,b]\setminus \mathcal{G}\exists i\le n [c_n<x<d_n]$.}\end{quote}

  The proof resembles the proof in  Subsubsection \ref{SSS:hb} but it is a bit more difficult.  Let $a,b$ in $\mathcal{R}$ be given such that $a<b$ and  let $\mathcal{G}$ be an open subset of $(a,b)$ such that $\mathcal{F}:=[a,b]\setminus \mathcal{G}$ is located. Note: $\{a,b\}\subseteq \mathcal{F}$.

  We define a function $E$ associating to every $s$ in $Bin$ a pair $E(s) =\bigl(E_0(s), E_1(s)\bigr)$ of real numbers, such that $\exists x \in \mathcal{F}[E_0(s) \le x \le E_1(s)]$. 
The definition is by induction to $\mathit{length}(s)$. We first define    $E\bigl(( \;)\bigr) := (a,b)$. Now assume  $s \in Bin$ and $E(s) := (r,u)$ has been defined.  Define $L(r,u):=(r, \frac{r+u}{2})$ and $R(r,u) :=(\frac{r+u}{2}, u)$, the \textit{left half} and the \textit{right half} of $(r,u)$, respectively.  Also define $L^+(r,u):=\bigl(r-\frac{1}{12}(r-u), \frac{r+u}{2}+\frac{1}{12}(r-u)\bigr)$ and $R^+(r,u):=\bigl(\frac{r+u}{2}-\frac{1}{12}(r-u), u+\frac{1}{12}(r-u)\bigr)$, the \textit{extended left half} and the \textit{extended right half} of $(r,u)$, respectively. 

We may decide: either $d( \frac{3}{4}r + \frac{1}{4}u, \mathcal{F}) > \frac{1}{4}(u-r)$ or: $d( \frac{3}{4}r + \frac{1}{4}u, \mathcal{F}) < \frac{1}{3}(u-r)$, and we  may also decide: either $d( \frac{1}{4}r + \frac{3}{4}u, \mathcal{F}) > \frac{1}{4}(u-r)$ or: $d( \frac{1}{4}r + \frac{3}{4}u, \mathcal{F}) < \frac{1}{3}(u-r)$. 

Note: if both $d( \frac{3}{4}r + \frac{1}{4}u, \mathcal{F}) > \frac{1}{4}(u-r)$ and $d( \frac{1}{4}r + \frac{3}{4}u, \mathcal{F}) > \frac{1}{4}(u-r)$, then $\neg \exists x \in \mathcal{F}[r\le x \le u]$, so this can not happen. We distinguish three cases.

 If we first discover  $d( \frac{3}{4}r + \frac{1}{4}u, \mathcal{F}) < \frac{1}{3}(u-r)$ and $d( \frac{1}{4}r + \frac{3}{4}u, \mathcal{F}) < \frac{1}{3}(u-r)$, we define: $E(s\ast\langle 0\rangle):= L^+(r,u)$ and $E(s\ast\langle 1\rangle):= R^+(r,u)$.

If we first discover  $d( \frac{3}{4}r + \frac{1}{4}u, \mathcal{F}) < \frac{1}{3}(u-r)$ and $d( \frac{1}{4}r + \frac{3}{4}u, \mathcal{F}) > \frac{1}{4}(u-r)$, we define: $E(s\ast\langle 0\rangle):= E(s\ast\langle 1\rangle) := L^+(r,u)$.
    
   If we first discover  $d( \frac{3}{4}r + \frac{1}{4}u, \mathcal{F}) > \frac{1}{4}(u-r)$ and $d( \frac{1}{4}r + \frac{3}{4}u, \mathcal{F}) < \frac{1}{3}(u-r)$, we define: $E(s\ast\langle 0\rangle):= E(s\ast\langle 1\rangle) := R^+(r,u)$. 
   
   Note: for both $i\le 1$, $E_1(s\ast\langle i \rangle ) - E_0(s\ast\langle i \rangle) = \frac{2}{3}\bigl(E_1(s) - E_0(s)\bigr)$. 
This completes the   definition of the function $E$.

 Let $\phi$ be a function from Cantor space $\mathcal{C}$ to $[a,b]$ such that, for every $\alpha$ in $\mathcal{C}$, for every $n$, $E_0(\overline \alpha n)<\phi(\alpha)<E_1(\overline \alpha n)$. One may verify  that $\phi$ is a well-defined surjective map from $\mathcal{C}$ onto $\mathcal{F}$.
 
 One  may complete the argument as in Subsubsection \ref{SSS:hb}.

  \subsubsection{Bar Induction}
  We also need \textit{Brouwer's Thesis on Bars}, or: \textit{Brouwer's Principle of Monotone Bar Induction}. In  \cite{brouwer24}, \cite{brouwer27} and \cite{brouwer54},  Brouwer used this principle for proving the Fan Theorem. 
  The Principle of Monotone Bar Induction is much stronger than the Fan Theorem.
  
  \smallskip Let $\mathbb{N}^\ast$ be the set of all finite sequences $c=\bigl(c(0), c(1), \ldots,c(n-1)\bigr)$ of natural numbers, where $n = \mathit{length}(c)$.   The empty sequence $(\;)$ is one of the elements of $\mathbb{N}^\ast$. For all $c=\bigl(c(0), c(1), \ldots,c(n-1)\bigr),d=\bigl(d(0), d(1), \ldots,d(p-1)\bigr)$ in $\mathbb{N}^\ast$, $c\ast d$ is the element of $\mathbb{N}^\ast$ that we obtain by putting $d$ behind $c$: $c\ast d= \bigl(c(0), c(1), \ldots,c(n-1),d(0), d(1), \ldots,d(p-1)\bigr)$. 
  
  Let $\mathcal{N}$ be the set of all infinite sequences $\alpha=\alpha(0), \alpha(1), \dots$ of natural numbers. For every $\alpha$ in $\mathcal{N}$, for every $n$, we define: $\overline \alpha n :=\bigl(\alpha(0), \alpha(1), \ldots,\alpha(n-1)\bigr)$. A subset $B$ of $\mathbb{N}^\ast$ is called a \textit{bar (in $\mathcal{N}$)} if and only if $\forall \alpha \in \mathcal{N} \exists n[\overline \alpha n \in B]$. A subset $B$ of $\mathbb{N}^\ast$ is called \textit{monotone} if and only if $\forall c[c\in B \rightarrow \forall m[c\ast( m) \in B]]$. A subset $C$ of $\mathbb{N}^\ast$ is called \textit{inductive} if and only if $\forall c[\forall m[c\ast(m) \in C]\rightarrow c \in C]$. Brouwer's Principle of Induction on Monotone Bars says the following, see \cite{kleenevesley65}, $ ^\ast 27.13$ and \cite{veldman06}. \begin{quote}{\it For all subsets $B,C$ of $\mathbb{N}^\ast$, if $B$ is monotone and a bar in $\mathcal{N}$, and $B\subseteq C$ and $C$ is inductive, then $( \;)\in C$.} \end{quote} \begin{theorem}\label{T:ccu} Let $\mathcal{G}$ be an open subset of $(-\pi,\pi)$. If $[-\pi, \pi]\setminus \mathcal{G}$ is  located and almost-enumerable, then $\mathcal{G}$ is eventually full. 
      
       \end{theorem}
     \begin{proof}
     
     Define $\mathcal{F} := [-\pi, \pi]\setminus \mathcal{G}$ and let
     $f$ be a function from $\mathbb{N}$ to $\mathcal{R}$ that almost-enumerates $\mathcal{F}$. 
      Note: $\forall x \in \mathcal{F} \forall \gamma \exists n[|f(n) - x| \le \frac{1}{2^{\gamma(n)}}]$, that is: $\forall \gamma \forall x \in \mathcal{F} \exists n[|f(n) - x| \le \frac{1}{2^{\gamma(n)}}]$. Using  the extended Heine-Borel Theorem, see Subsubsection \ref{SSS:hblocated}, we conclude: $$\forall \gamma \in \mathcal{N} \exists N\forall x \in \mathcal{F} \exists n\le N[|f(n) - x| \le \frac{1}{2^{\gamma(n)}}].$$
      
      We now let $B$ be the set of all $c$ in $\mathbb{N}^\ast$ such that $\forall x \in \mathcal{F} \exists n< \mathit{length}(c) [|f(n) - x| \le \frac{1}{2^{c(n)}}]$. Note: $B$ is a bar in $\mathcal{N}$ and, clearly, $B$ is monotone.
      
      \smallskip
      For every $c$ in $\mathbb{N}^\ast$, we let $\mathcal{H}_c$ be the set of all $x$ in $[-\pi,\pi]$ such that $\forall n< \mathit{length}(c) [|f(n) - x| > \frac{1}{2^{c(n)}}]$. We let $C$ be the set of all $c$ in $
     \mathbb{N}^\ast$ such that, for some $\mathcal{X}$ in $Ext_\mathcal{G}$, $\mathcal{H}_c \subseteq \mathcal{X}$.
      
      Note: for all $c$ in $\mathbb{N}^\ast$, if $c \in B$, then $\mathcal{H}_c = \emptyset$, and, therefore, $c\in C$. We thus see: $B\subseteq C$.
      
      \smallskip
      Assume: $c\in \mathbb{N}^\ast$ and, for all $m$, $c\ast( m ) \in C$. Find $n:=\mathit{length}(c)$. Find $\mathcal{X}_0, \mathcal{X}_1, \ldots$ in $Ext_\mathcal{G}$ such that, for each $m$, $\mathcal{H}_{c\ast( m )}=\mathcal{H}_c \cap \{x \in [-\pi, \pi]||x-f(n)|>\frac{1}{2^m}\} \subseteq \mathcal{X}_m$. We now prove: $\mathcal{H}_c \subseteq  (\bigcup\limits_{p\in \mathbb{N}} \mathcal{X}_p)^+$.
      
       Assume: $x\in\mathcal{H}_c$. Find $m$ such that $(x-\frac{1}{2^m}, x+\frac{1}{2^m})\subseteq \mathcal{H}_c$ and both: $x-\frac{1}{2^m}\;\#_\mathcal{R}\; f(n)$ and $x+\frac{1}{2^m}\;\#_\mathcal{R}\; f(n)$.  Distinguish two cases. 
       
       \textit{Case (i)}. $f(n)<x-\frac{1}{2^m}$ or  $f(n)>x+\frac{1}{2^m}$. Then $|x-f(n)|>\frac{1}{2^m}$ and $x \in \mathcal{X}_m\subseteq (\bigcup\limits_{p\in \mathbb{N}} \mathcal{X}_p)^+$. 
       
        \textit{Case (ii)}. $x-\frac{1}{2^m}<f(n)<x+\frac{1}{2^m}$. Find $q$ such that  $x-\frac{1}{2^m}<f(n)-\frac{1}{2^q}<f(n) +\frac{1}{2^q}<x+\frac{1}{2^m}$. Note: for each $p>q$, $(x-\frac{1}{2^m}, f(n) -\frac{1}{2^p})\subseteq \mathcal{X}_p$  and  $\bigl(x -\frac{1}{2^m}, f(n)\bigr)= \bigcup\limits_{p>q}(x-\frac{1}{2^m}, f(n) -\frac{1}{2^p}) \subseteq \bigcup\limits_{p \in \mathbb{N}}\mathcal{X}_p$. Also, for each $p>q$, $(f(n) +\frac{1}{2^p}, x+\frac{1}{2^m}) \subseteq \mathcal{X}_p$ and    $\bigl(f(n), x +\frac{1}{2^m}\bigr)=\bigcup\limits_{p>q}(f(n)+\frac{1}{2^p}, x +\frac{1}{2^m}) \subseteq \bigcup\limits_{p \in \mathbb{N}}\mathcal{X}_p$. Conclude: $\bigl(x-\frac{1}{2^m}, f(n)\bigr)\cup\bigl(f(n), x +\frac{1}{2^m}\bigr)\subseteq  \bigcup\limits_{p \in \mathbb{N}}\mathcal{X}_p$ and  $x \in (x-\frac{1}{2^m}, x+\frac{1}{2^m})\subseteq (\bigcup\limits_{p \in \mathbb{N}}\mathcal{X}_p)^+$.
        
         Conclude: $\mathcal{H}_c\subseteq (\bigcup\limits_{p\in \mathbb{N}} \mathcal{X}_p)^+ \in Ext_\mathcal{G}$, and $c \in C$. 
        
       We have shown: for all $c$ in $\mathbb{N}^\ast$,\textit{ if $\forall m[c\ast( m) \in C]$, then $c\in C$}, that is: $C$ is inductive.
      
      \smallskip We thus see: $B$ is a bar in $\mathcal{N}$, $B$ is monotone, $B\subseteq C$ and $C$ is inductive. Using the Principle of Induction on Monotone Bars, we conclude: 
       $(\;) \in C$, so there exists $\mathcal{X}$ in $Ext_\mathcal{G}$ such that $(-\pi,\pi) = \mathcal{H}_{(\;)} \subseteq \mathcal{X}$. Conclude: $(-\pi,\pi)\in Ext_\mathcal{G}$, that is: $\mathcal{G}$ is eventually full. \end{proof}
   It is possible to prove a Theorem that is a little bit stronger than Theorem \ref{T:ccu}. One obtains this Theorem from Theorem \ref{T:ccu} by replacing the conclusion: `\textit{$\mathcal{G}$ is eventually full}'  by the  statement: `\textit{there exists $\mathcal{H}$ in $\mathcal{CBO}_{(-\pi,\pi)}$ such that $\mathcal{H} \subseteq \mathcal{G}$}'.  As every $\mathcal{H}$ in  $\mathcal{CBO}_{(-\pi,\pi)}$ is eventually full, this statement implies that $\mathcal{G}$ itself is eventually full.  
   
   \section{Every co-enumerable subset of $[-\pi,\pi]$ guarantees uniqueness}   A subset $\mathcal{X}$ of $\mathcal{R}$ is \textit{co-enumerable} if and only if there exists a function $f$ from $\mathbb{N}$ to $\mathcal{R}$ such that,  for all $x$ in $\mathcal{R}$, if $\forall n[x\;\#_\mathcal{R}\; f(n)]$, then $x \in \mathcal{X}$. \subsection{An intuitionistic proof of an extended  Cantor-Schwarz-Lemma}We first prove two preliminary Lemmas. 
   
   \begin{lemma}\label{L:helpcsby} Let $a,b$ in $\mathcal{R}$ be given such that $a<b$ and let $H$ be a function from $[a,b]$ to $\mathcal{R}$. Assume: $a<z<b$ and $\forall y\in[a,b][H(y) \le H(z)]$, that is, $H$ assumes its greatest value at $z$. Also assume: $D^1H(z)=0$. Then $H$ is differentiable at $z$ and $H'(z) =0$. \end{lemma}
   
   \begin{proof}
     Note:  $$D^1H(z)= \lim_{h \rightarrow 0} \frac{H(z+h) + H(z-h)-2H(z)}{h} .$$ 
  
  As $H$ assumes its greatest value at $z$,   for every $h$,  $2H(z) - H(z+h)-H(z-h) = \bigl(H(z)-H(z+h)\bigr) + \bigl(H(z)-H(z-h)\bigr)\ge H(z)-H(z+h)\ge 0.$
  
   Conclude: $\lim_{h \rightarrow 0} \frac{H(z+h) -H(z)}{h} =0$, that is, $H$ is differentiable at $z$ and \\$H'(z)=0$.
    \end{proof}
    \begin{lemma}\label{L:helphelpcsby} Let $a,b$ in $\mathcal{R}$ be given such that $a<b$ and let $G$ be a continuous function from $[a,b]$ to $\mathcal{R}$. \begin{enumerate}[\upshape (i)] \item For all $y,z$ in $[a,b]$, if $G(y) \;\#_\mathcal{R} \;G(z)$, then $y\;\#_\mathcal{R}\; z$. 
    \item For all $y,z$ in $[a,b]$, if $G$ is differentiable at both $y$ and $z$ and $G'(y)\;\#_\mathcal{R}\;G'(z)$, then $y\;\#_\mathcal{R} \;z$. \end{enumerate}  \end{lemma}
    
    \begin{proof} (i) Assume: $G(y) \;\#_\mathcal{R}\;G(z)$ and define: $\varepsilon=|G(y) - G(z)|$. Use the fact that $G$ is continuous at $y$ and find $\delta$ such that $\forall v \in[a,b][|y-v|<\delta \rightarrow |G(y) - G(v)|<\varepsilon]$. Conclude: $|y-z|\ge \delta$, and: $y\;\#_\mathcal{R} \; z$. 
    
    \smallskip (ii) Assume:  $G$ is differentiable at both $y$ and $z$ and $G'(y)\;\#_\mathcal{R}\;G'(z)$. Define: $\varepsilon=|G'(y) - G'(z)|$. Find $h>0$ such that both $|G'(y)-\frac{G(y+h)-G(y)}{h}| <\frac{\varepsilon}{3}$ and  $|G'(z)-\frac{G(z+h)-G(z)}{h}| <\frac{\varepsilon}{3}$. Conclude: $|\frac{G(y+h)-G(y)}{h}-\frac{G(z+h)-G(z)}{h}|>\frac{\varepsilon}{3}$ and $|\bigl(G(y+h)-G(y)\bigr)-\bigl(G(z+h)-G(z)\bigr)|>\frac{\varepsilon}{3}h$ and \textit{either}: $|G(y+h)-G(z+h)|>\frac{\varepsilon}{6}h$, and therefore, by (i), $y+h \;\#_\mathcal{R}\;z+h$ and thus $y \;\#_\mathcal{R}\;z$, \textit{or}: $|G(y) - G(z)|>\frac{\varepsilon}{6}h$ and, again by (i), $y \;\#_\mathcal{R}\;z$.  \end{proof}
   \begin{lemma}[Cantor-Schwarz-Bernstein-Young]\label{L:csby}Let $a<b$ be given and let $f$ be a function from $\mathbb{N}$ to $\mathcal{R}$.  Let $\mathcal{X}$ be the set of all $x$ in $(a, b)$ such that, for all $n$, $x \;\#_\mathcal{R}\;f(n)$. Let  $G$ be a function from $[a,b]$ to $\mathcal{R}$ such that $G(a) = G(b) =0$ and,  for all $x$ in $\mathcal{X}$, $D^2G(x)$ exists. \begin{enumerate}[\upshape (i)] \item For each $\varepsilon >0$, if $\exists x \in [a,b][G(x)=\varepsilon]$, then $\exists x \in
      \mathcal{X}[D^2G(x) \le -2\varepsilon]$, and,  \item if $\forall x \in \mathcal{X}[D^2G(x) = 0]$, then $\forall x \in [a,b][G(x) = 0]$. \end{enumerate}\end{lemma}\begin{proof}(i) Assume we find $x$ in $[a,b]$ such that $G(x) =\varepsilon>0$. Define a function $H$ from $[a,b]$ to $\mathcal{R}$ such that, for all $y$ in $[a,b]$, $H(y)= G(y) -\varepsilon\frac{(b-y)(y-a)}{(b-a)^2}$. Note: $H(x) \ge \frac{3}{4}\varepsilon$ and, for all $y$ in $\mathcal{X}$, $D^2H(y)$ exists and $D^2H(y) = D^2G(y)+ 2\varepsilon$.
      
       For each real number $\rho$ we 
   define a function $H_\rho$ from $[a,b]$ to $\mathcal{R}$ such that, for all $y$ in $[a,b]$, $H_\rho(y) = H(y) + \rho \frac{y-a}{b-a}$.  Note: for each $\rho$, for all $y$ in $\mathcal{X}$, $D^2H_\rho(y)$ exists and $D^2H_\rho(y) = D^2H(y) = D^2G(y)+2\varepsilon$. Also note: for all $\rho$ in $[0, \frac{1}{2}\varepsilon]$, $H_\rho(x) > H_\rho(a)$ and $H_\rho(x) > H_\rho(b)$.  
   
   We now have to re-read the proof of Lemma \ref{L:cs}. The basic step in the inductive construction there was the following:  
   
   \begin{quote}
   Let $(v,w), (c,d)$ be  pairs of reals such that $v<w$ and $c<d$.  Define $z_0 := \frac{2}{3}v +\frac{1}{3}w$ and  $z_1:= \frac{1}{3}v +\frac{2}{3}w$. Then there exist $c_0, d_0$ such that $c\le c_0<d_0\le d$ and \textit{either}: $\forall \rho \in [c_0, d_0][\sup_{[v, z_1]}H_\rho = 
   \sup_{[v, w]}H_\rho]$ \textit{or}: $\forall \rho \in [c_0, d_0][\sup_{[z_0, w]}H_\rho = 
   \sup_{[v, w]}H_\rho]$. 
   \end{quote}
 A slight modification of the argument (we choose two disjoint intervals within the interval $(c_0, d_0)$) leads to  the following conclusion: 
   \begin{quote} Let $(v,w), (c,d)$ be  pairs of reals such that $v<w$ and $c<d$.  Define $z_0 := \frac{2}{3}v +\frac{1}{3}w$ and  $z_1:= \frac{1}{3}v +\frac{2}{3}w$. Then there exist $c_0, d_0, c_1,d_1$ such that $c\le c_0<d_0<c_1<d_1\le d$ and $\forall i<2[d_i-c_i<\frac{1}{2}(d-c)]$ and \textit{either} (case (a)): $\forall i<2\forall \rho \in [c_i, d_i][\sup_{[v, z_1]}H_\rho = 
   \sup_{[v, w]}H_\rho]$ \textit{or} (case (b)): $\forall i<2\forall \rho \in [c_i, d_i][\sup_{[z_0, w]}H_\rho = 
   \sup_{[v, w]}H_\rho]$. 
   \end{quote}
   
   We may use this to define two functions, called $D, F$ from the set $Bin$ of the finite binary sequences to the set of the pairs of real numbers.
   
   \begin{enumerate}[\upshape 1.] \item Define $D\bigl(( \;)\bigr) := (a,b)$ and $F\bigl((\;)\bigr):=(0,\frac{1}{2}\varepsilon)$ \item Assume $s \in Bin$ and $D(s), F(s)$ have been defined. Find  $v,w,c,d$  such that $D(s) = (v,w)$ and $F(s) =(c,d)$.  Apply the above construction and define: $F(s\ast( 0)) = (c_0, d_0)$ and $F(s\ast( 1)) = (c_1, d_1)$, and, in case (a), $D(s\ast( 0))= D(s\ast( 1))= (v, z_1)$, and, in case (b), $D(s\ast( 0))=D(s\ast( 1)) = (z_0, w)$. \end{enumerate}
   Let us write, for each $s$ in $Bin$, $D(s)=\bigl(D_0(s), D_1(s)\bigr)$ and $F(s)=\bigl(F_0(s), F_1(s)\bigr)$.
   
   Define  functions $\phi$, $\psi$ from Cantor space $\mathcal{C}$ to   $[a,b]$ and $ [0, \frac{1}{2}\varepsilon]$, respectively, such that, for all $\alpha$, for all $n$, $D_0(\overline \alpha n)\le\phi(\alpha) \le D_1(\overline \alpha n)$ and $F_0(\overline \alpha n)\le\psi(\alpha) \le F_1(\overline \alpha n)$.
   
   The following two conclusions should be clear:
     \begin{enumerate}[\upshape (1)] \item for all $\alpha, \beta$ in $\mathcal{C}$, if $\alpha \;\#\;\beta$, then $\psi(\alpha) \;\#_\mathcal{R}\;\psi(\beta)$, that is: \textit{the function $\psi$ is strongly injective}, and,  \item for all $\alpha $ in $\mathcal{C}$, for all $y$ in $[a,b]$, $H_{\psi(\alpha)}(y) \le H_{\psi(\alpha)}\bigl(\phi(\alpha)\bigr)$, that is: the function  $H_{\psi(\alpha)}$ assumes its greatest value at $\phi(\alpha)$. \end{enumerate}
  
  We now prove: \textit{the function $\phi$ is strongly injective}. 
  
  By Riemann's second result,   $0=D^1G\bigl(\phi(\alpha)\bigr)=D^1H_{\psi(\alpha)}\bigl(\phi(\alpha)\bigr)$.
  Also,  the function $H_{\psi(\alpha)}$ assumes its greatest value at $\phi(\alpha)$. Using Lemma \ref{L:helpcsby}, we conclude: the function $H_{\psi(\alpha)}$ is differentiable at $\phi(\alpha)$ and $(H_{\psi(\alpha)})'\bigl(\phi(\alpha)\bigr)=0$. It follows that $G$ itself is differentiable at $\phi(\alpha)$ and that $G'\bigl(\phi(\alpha)\bigr) = -\frac{\psi(\alpha)}{b-a}$.

    For all $\alpha, \beta$ in $\mathcal{C}$, if $\alpha \;\#\;\beta$, then $\psi(\alpha) \;\#_\mathcal{R}\; \psi(\beta)$, so $G'\bigl(\phi(\alpha)\bigr) \;\#_\mathcal{R}\;G'\bigl(\phi(\beta)\bigr)$ and, therefore, by Lemma \ref{L:helphelpcsby}(ii): $\phi(\alpha) \;\#_\mathcal{R}\; \phi(\beta)$.
    
    \smallskip
    It follows that the set $\{\phi(\alpha)|\alpha \in \mathcal{C}\}$ is \textit{positively uncountable} in the following sense: given any function $g$ from $\mathbb{N}$ to $\mathcal{R}$, one may build $\alpha$ in $\mathcal{C}$ such that $\forall n[g(n) \;\#_\mathcal{R}\; \phi(\alpha)]$.
    
     We  do so for the function $f$ from $\mathbb{N}$ to $\mathcal{R}$ that occurs in the data of our Theorem. 
     
      We  build the promised $\alpha $ in $\mathcal{C}$ step by step. Together with $\alpha$ we construct a strictly increasing element $\zeta$ of $\mathcal{N}$ and we will take care that, for each $n$, either $f(n) <_\mathcal{R} D_0\bigl(\overline \alpha\zeta(n)\bigr)$ or $D_1\bigl(\overline \alpha \zeta(n)\bigr)<_\mathcal{R} f(n)$. 
     
     We define $\zeta(0) = 0$. Now let $n$ be given and assume we constructed $\overline \alpha \zeta(n)$ successfully. We define $\beta :=\overline\alpha \zeta(n)\ast\underline 0$ and $\gamma:=\overline \alpha \zeta(n) \ast\underline 1$. Note: $\beta\;\#\;\gamma$ and, therefore: $\phi(\beta) \;\#_\mathcal{R}\;\phi(\gamma)$. Find $p$ such that either $D_1(\overline \beta p)<_\mathcal{R} D_0(\overline \gamma p)$ or $D_1(\overline \gamma p)<_\mathcal{R} D_0(\overline \beta p)$. 
 Now distinguish two cases.
     
     \textit{Case (i)}: $D_1(\overline \beta p)<_\mathcal{R} D_0(\overline \gamma p)$. Then either: $f(n) <_\mathcal{R} D_0(\overline \gamma p)$ or $D_1(\overline \beta p)<f(n)$. In case we first find out: $f(n) <_\mathcal{R} D_0(\overline \gamma p)$, we define: $\zeta(n+1): =p$ and $\overline \alpha \zeta(n+1) = \overline \gamma p$, and in case we first find out: $D_1(\overline \beta p)<f(n)$, we define: $\zeta(n+1): =p$ and $\overline \alpha \zeta(n+1) = \overline \beta p$.
     
     \textit{Case (ii)}: $D_1(\overline \gamma p)<_\mathcal{R} D_0(\overline \beta p)$.  This case is handled similarly. (Interchange the r\^oles of $\beta$ and $\gamma$.) 
   
   This completes the description of the construction of $\alpha$. 
    Define $x :=\phi(\alpha)$.
    
    Note: $\forall n[x \;\#_\mathcal{R} \;f(n)]$. Therefore: $x \in \mathcal{X}$ and  $D^2G(x)$ exists. Conclude: also  $D^2H_{\psi(\alpha)}(x)$ exists and $D^2H_{\psi(\alpha)}(x)= D^2G(x) +2\varepsilon$. But,  $H_{\psi(\alpha)}$ assumes its greatest value at $x$ and thus $D^2H_{\psi(\alpha)}(x)\le 0$. Therefore: $D^2G(x) \le -2\varepsilon$.
    
    \smallskip
   (ii) follows from (i), as in the proof of Lemma \ref{L:cs}. \end{proof}
    \begin{corollary}\label{cor:csby} Let $a,b$ be real numbers such that $a<b$  and let $G$ be a function from $[a,b]$ to $\mathcal{R}$. Let $\mathcal{X}$ be a co-enumerable subset of $[a,b]$ such that   for all $x$ in $\mathcal{X}$, $D^2G(x)=0$. Then $G$ is linear on $[a,b]$, that is: for all $x$ in $[a,b]$, $G(x)=G(a) +\frac{x-a}{b-a}\bigl(G(b)-G(a)\bigr)$. \end{corollary} 
  
   \begin{proof} Define, for each $x$ in $[a,b]$:  $$G^\ast(x):= G(x)-G(a) -\frac{x-a}{b-a}\bigl(G(b)-G(a)\bigr)$$ and conclude, using Lemma \ref{L:csby}(ii): for all $x$ in $[a,b]$, $G^\ast(x) = 0$. \end{proof}  
    \subsection{A second proof} 
 It seems to us this second proof is of interest although it does not give the constructive information of Lemma \ref{L:csby}(i). Lemma \ref{L:csby2} stands to Lemma   \ref{L:csby} as Lemma \ref{L:cs2} stands to Lemma \ref{L:cs}. 
    \begin{lemma}[Cantor-Schwarz-Bernstein-Young, version II]\label{L:csby2}Let $a<b$ be given and let $f$ be a function from $\mathbb{N}$ to $\mathcal{R}$.  Let $\mathcal{X}$ be the set of all $x$ in $(a, b)$ such that, for all $n$, $x \;\#_\mathcal{R}\;f(n)$. Let  $G$ be a function from $[a,b]$ to $\mathcal{R}$ such that $G(a) = G(b) =0$ and,  for all $x$ in $\mathcal{X}$, $D^2G(x)$ exists. If $\forall x \in \mathcal{X}[D^2G(x) = 0]$, then $\forall x \in [a,b][G(x) = 0]$. \end{lemma}

  \begin{proof} Assume we find $x$ in $[a,b]$ such that $G(x) >0$. Let  $H$ be the function from $[a,b]$ to $\mathcal{R}$ such that, for all $y$ in $[a,b]$, $H(y)= G(y) -G(x)\frac{(b-y)(y-a)}{(b-a)^2}$. Note: $H(x) \ge \frac{3}{4}G(x) >0$ and, for all $y$ in $(a,b)$, $D^2H(y)$ exists and $D^2H(y) = D^2G(y) +2G(x)>0$. 
  
 For each real number $\rho$,  
   let $H_\rho$ be the function from $[a,b]$ to $\mathcal{R}$ such that, for all $y$ in $[a,b]$, $H_\rho(y) = H(y) + \rho \frac{y-a}{b-a}$.  Note: for each $\rho$, for all $y$ in $(a,b)$, $D^2H_\rho(y)$ exists  and $D^2H_\rho(y) =D^2H(y) = D^2G(y) + 2G(x)$. Also note: for all $\rho$ in $[0, \frac{1}{2}G(x)]$, $H_\rho(x) =H(x) +\rho\frac{x-a}{b-a}\ge \frac{3}{4}G(x) >0= H_\rho(a)$ and $H_\rho(x)\ge\frac{3}{4}G(x) > \rho=H_\rho(b)$.
    
Find $\delta>0$ such that $\forall z \in [a, a+\delta)\cup(b-\delta,b][G(z) < \frac{1}{2}G(x)]$ and note: $\forall \rho \in[0, \frac{1}{2}G(x)]\forall z
\in [a, a+\delta)\cup(b-\delta,b][H_\rho(z) < H_\rho(x)]$.

We now intend to prove: $\exists \rho \in \mathcal{R}\forall z \in [a,b]\exists y \in[a,b][H_\rho(z) < H_\rho(y)]$. 

To this end, we define, step by step, an infinite sequence $(c_0, d_0), (c_1, d_1), \ldots$ of pairs of real numbers such that $(c_0, d_0) = (0, \frac{1}{2}G(x))$ and, for each $n$, \begin{enumerate}[\upshape (i)] \item  $c_n<c_{n+1} < d_{n+1} < d_n$ and $d_{n+1}-c_{n+1}<\frac{2}{3}(d_n -c_n)$, and \item  $\forall \rho \in (c_{n+1}, d_{n+1})\exists y \in [a,b][H\bigl(f(n)\bigr) < H(y)]$.
\end{enumerate}
Let $n$ be given such that $(c_n, d_n)$ has been defined already. We consider $f(n)$ and  distinguish two cases. 

\textit{Case (a)}. $f(n) \in [a, a+\delta)\cup(b-\delta,b]$. Then: $\forall \rho \in [0, \frac{1}{2}G(x)][H_\rho\bigl(f(n)\bigr) <H_\rho(x)]$. We define: $(c_{n+1}, d_{n+1}) := (\frac{2}{3}c_n +\frac{1}{3}d_n, \frac{1}{3}c_n +\frac{2}{3}d_n)$. 

\textit{Case (b)}. $a<f(n)<b
 $. Using the construction from the proof of Lemma \ref{L:cs}, find $\rho_0, \rho_1$ in $(\frac{2}{3}c_n +\frac{1}{3}d_n, \frac{1}{3}c_n +\frac{2}{3}d_n)$ and $z_0,z_1$ in $(a,b)$ such that  $\rho_0\;\#_\mathcal{R}\;\rho_1$ and  $\forall i<2 \forall y \in [a,b][y \;\#_\mathcal{R} \;z_i \rightarrow H_{\rho_i}(y) < H_{\rho_i}(z_i)]$.   As we saw in the proof of Lemma \ref{L:csby}, one may conclude: $G$ is differentiable at both $z_0$ and $z_1$ and $\forall i<2[G'(z_i) = -\frac{\rho_i}{b-a}]$ and, therefore, by Lemma \ref{L:helphelpcsby}, $z_0 \;\#\;z_1$.  
 Note: either: $f(n) \;\#_\mathcal{R}\; z_0$ or: $f(n) \;\#_\mathcal{R}\; z_1$.  We distinguish two subcases. 
 
 \textit{Case (b)i}.  $f(n) \;\#_\mathcal{R}\; z_0$. Conclude: $H_{\rho_0}\bigl(f(n)\bigr) < H_{\rho_0}(z_0)$. Consider $\varepsilon:= H_{\rho_0}(z_0)-H_{\rho_0}\bigl(f(n)\bigr)$ and define: 
 
 $(c_{n+1}, d_{n+1}):=\bigl(\sup(\frac{2}{3}c_n +\frac{1}{3}d_n, \rho_0 -\frac{\varepsilon}{3}), \inf(\rho_0+\frac{\varepsilon}{3}, \frac{1}{3}c_n +\frac{2}{3}d_n)\bigr)$. 
 
 Then, for each $\rho$ in $(c_{n+1}, d_{n+1})$, $|H_{\rho}\bigl(f(n)\bigr)-H_{\rho_0}\bigl(f(n)\bigr)|<\frac{\varepsilon}{3}$ and $|H_\rho(z_0) -H_{\rho_0}(z_0)|< \frac{\varepsilon}{3}$ and: $H_\rho\bigl(f(n)\bigr) < H_\rho(z_0)$. 
 
 \textit{Case (b)ii}.  $f(n) \;\#_\mathcal{R}\; z_1$. This subcase is treated like subcase (b)i. (Replace everywhere the subindex $0$ by the subindex $1$.)
 
 Now find $\rho$ such that $\forall n[c_n < \rho< d_n]$ and note: $\forall n \exists y \in [a,b][H_\rho\bigl(f(n)\bigr) < H_\rho(y)]$. Also observe: for all $z$ in $(a,b)$, if $\forall n[f(n) \;\#_\mathcal{R} \; z]$, then $D^2H_\rho(z) >0$, and, according to Lemma \ref{L:basic}, $\exists y \in [a,b][H_\rho(z)<H_\rho(y)]$. The statement $\forall z \in (a,b)[\exists n[f(n) = x]\;\vee\;\forall n[f(n) \;\#_\mathcal{R}\; z]]$ is false, but, nevertheless, one may prove: $\forall z \in (a,b) \exists y \in [a,b][H_\rho(z) < H_\rho(y)]$, as we do now.
 
  First, find, using the continuity of $H_\rho$ an infinite sequence $\delta_0, \delta_1, \dots$ of  reals such that, for each $n$, $0<\delta_n<\frac{1}{2^n}$ and $ \forall z \in(a,b)[|f(n) - z| <\delta_n \rightarrow \exists y \in[a,b][H_\rho(z) < H_\rho(y)]]$. Now, let $z$ be an element of $(a,b)$. Find $\alpha$ such that, for each $n$, if $\alpha(n) =0$, then $|f(n) -z|>\frac{1}{2}\delta_n$, and, if $\alpha(n) \neq 0$, then $|f(n) -z|<\delta_n$. Define an infinite sequence of pairs of reals $(a_0, b_0), (a_1, b_1), \ldots$ such that, for each $n$ \begin{enumerate}[\upshape (i)] \item if $\forall i\le n[\alpha(i) = 0]$, then $(a_n, b_n) = (z-\frac{1}{2}\delta_n, z+\frac{1}{2}\delta_n)$, and, \item if $\exists i \le n[\alpha(i) \neq 0]$,  then $a_{n-1} <a_n <b_n<b_{n-1}$ and $b_n -a_n < \frac{1}{2}(b_{n-1}-a_{n-1})$ and $f(n)<a_n$ or $b_n < f(n)$.
  \end{enumerate}
  Find $z^\ast$ such that $\forall n[a_n < z^\ast < b_n]$. Note $\forall n[f(n) \;\#_\mathcal{R} \;z^\ast]$ and find $y$ such that $H_\rho(z^\ast) < H_\rho(y)$.  Find $m$  such that  $\forall v \in {a,b}[|z^\ast -v|<\frac{1}{2^m} \rightarrow H_\rho(v) < H_\rho(y)]$ and distinguish two cases.
  
  \textit{Case (a)}. $|z^\ast-z|<\frac{1}{2^m}$. Conclude: $H_\rho(z) < H_\rho(y)$.
  
  \textit{Case (b)}. $|z^\ast -z|>\frac{1}{2^{m+1}}> \frac{1}{2}\delta_m$. Find $i\le m$ such that $\alpha(i) \neq 0$, and, therefore: $|f(n) -z| < \delta_i$ and: $\exists y \in [a,b]][H_\rho(z) < H_\rho(y)]$.
  
  We thus see: $\forall z \in (a,b)\exists y \in[a,b][H_\rho(z) < H_\rho(y)]$. 
 In the beginning of the proof we found $\delta>0$  such that $\forall \rho \in[0, \frac{1}{2}G(x)]\forall z
\in [a, a+\delta)\cup(b-\delta,b][H_\rho(z) < H_\rho(x)]$.
We may conclude:  $\forall z \in [a,b]\exists y \in[a,b][H_\rho(z) < H_\rho(y)]$. That is impossible, according to Theorem \ref{T:weierstr}.

We have to conclude: $\neg \exists x \in[a,b][G(x) >0]$. In a similar way, we prove:  $\neg \exists x \in[a,b][G(x) <0]$. Therefore: $\forall x \in[a,b][G(x) =0]$.
  \end{proof}
\subsection{The final result}\begin{theorem}\label{T:triumph} Every co-enumerable subset $\mathcal{X}$ of $[\pi,\pi]$ guarantees uniqueness.\end{theorem} \begin{proof} Find a function $f$ from $\mathbb{N}$ to $\mathcal{R}$ such that, for each $x$ in $[-\pi,\pi]$, if for all $n$, $x\;\#_\mathcal{R}\;f(n)$, then $x\in \mathcal{X}$. Let $b_0, a_1, b_1, \ldots$ be an infinite sequence of reals such that,  for all $x$ in $\mathcal{X}$, $$F(x) := \frac{b_0}{2}+\sum_{n>0} a_n \sin nx + b_n \cos nx =0.$$

   \textit{We make a provisional assumption}:   the infinite sequence $b_0, a_1, b_1, \ldots$ is bounded.
   
   The function $$G(x) := \frac{1}{4}b_0x^2 +\sum_{n>0}\frac{- a_n}{n^2} \sin nx + \frac{-b_n}{n^2} \cos nx$$ is therefore defined everywhere on $[-\pi,\pi]$ and everywhere continuous, 
    and, according to Lemma \ref{L:csby}, for all $x$ in $\mathcal{X}$, $D^2G(x) =0$. Use Corollary \ref{cor:csby} and conclude, as in the proof of Theorem \ref{T:cu},  $b_0=0$ and, for all $n>0$, $a_n = b_n =0$.

     \smallskip \textit{One may do without the provisional assumption}. We no longer assume that    the infinite sequence $b_0, a_1, b_1, \ldots$ is bounded.
   
   Using again the suggestion made by Riemann and  Kronecker, we reason as follows.
   
    Assume: $x \in \mathcal{X}$. Define for each $t$ in $[-\pi, \pi]$, $$ K(t) := F(x+t) + F(x-t)$$ and note: $$K(t) = b_0 + 2\sum_{n>0}(a_n\sin nx + b_n \cos nx)\cos nt,$$

   Define a function $g$ from $\mathbb{N}$ to $\mathcal{R}$ such that, for each $n$, $g(2n) := f(n)-x$ and  $g(2n+1):= = f(n) +x $. Note: for each $t$ in $[-\pi,\pi]$, if, for all $n$, $g(n) \;\#_\mathcal{R}\; t$, then   $K(t)=0$, and: the sequence $n \mapsto a_n\sin nx + b_n \cos nx$ converges and is bounded. Using the first part of the proof, we conclude: $b_0 = 0$ and, for each $n>0$,  $ a_n\sin nx + b_n \cos nx=0$. This conclusion holds for all $x$  in $\mathcal{X}$.  As $\mathcal{X}$ is a co-enumerable subset of $[-\pi,\pi]$, we may conclude: $b_0 = 0$ and, for each $n>0$, $a_n = b_n = 0$. \end{proof}
   
   \subsection{Some comments} There is irony in history.
   
   If we restrict ourselves to co-enumerable subsets $\mathcal{X}$ of $(-\pi, \pi)$ such that $[\pi,\pi]\setminus \mathcal{X}$ is located and  almost-enumerable, then Theorem \ref{T:triumph} is a much stronger statement than Theorem \ref{T:ccu}. The stronger and later result is obtained by more simple means.  We did not use Brouwer's Principle of Induction on Monotone Bars, as we did in the proof of Theorem \ref{T:ccu}. Every reference to ordinals or generalized inductive definitions has disappeared. Cantor's original problem has been solved without set-theoretic means.

  The classical version of Theorem \ref{T:triumph} is due to F.~Bernstein, see \cite{bernstein08}, and  W.~Young, see \cite{young09}. One wonders what Cantor himself has thought or would have thought about the result by Bernstein and Young. Although it was obtained during his lifetime, I suspect he has not been able to give any comment.
  
  In \cite{bary64}, chapter XIV, Section 5, the result occurs as a corollary of a theorem due to du Bois-Reymond.  Bernstein concluded from the extended Cantor-Schwarz Lemma \ref{L:csby}  that every \textit{totally imperfect} set, that is, a set without a perfect subset, is a set of uniqueness, see also \cite{cooke93}.
 \section{Using Brouwer's Continuity Principle}\label{S:bcp}
 
 \textit{Brouwer's Continuity Principle} says the following. \begin{quote} {\it Let $R$ be a subset of $\mathcal{N}\times\mathbb{N}$.
 
 If  $\forall \alpha \exists n[\alpha R n]$, then  $\forall \alpha \exists m \exists n\forall \beta[\overline \beta m = \overline \alpha m \rightarrow \beta Rn]$.} \end{quote} 
 
 Brouwer's Continuity Principle has the following consequence: 
  \begin{quote} {\it Let $a, b$ in $\mathcal{R}$ be given such that $a<b$ and let $R$ be a \emph{real} subset of $[a,b]\times\mathbb{N}$, that is: $\forall n \forall x \in [a,b]\forall y \in [a,b][(x=_\mathcal{R}y \;\wedge\; xRn )\rightarrow yRn]$.
  
 If  $\forall x \in[a,b] \exists n[x R n]$, then  \\$\forall x \in [a,b] \exists c\exists d[c<x<d\;\wedge\; \exists n\forall y \in [c,d]\cap[a,b][ y Rn]]$.} \end{quote} 
 
 One may prove this using the fact that there exists a continuous surjection from $\mathcal{N}$ onto $[a,b]$.

 We let $[\omega]^\omega$ denote the set of all $\zeta$ in $\mathcal{N}$ such that $\forall n[\zeta(n)<\zeta(n+1)]$. 
 
 \subsection{The Cantor-Lebesgue Theorem}
 
 \begin{lemma}[Cantor's nagging question]\label{L:cl} Let $a, b, a_0, b_0, a_1, \ldots$ be an infinite sequence of reals such that $a<b$ and, for all $x$ in $[a,b]$, $ \lim_{n \rightarrow \infty} a_n \sin nx + b_n \cos nx = 0$, that is: $\forall p\forall x \in [a,b]\exists n\forall m \ge n[|a_m \sin mx + b_m \cos mx|<\frac{1}{2p}]$. Then: \begin{enumerate}[\upshape (i)] \item (using Brouwer's Continuity Principle):\\ $\lim_{n \rightarrow \infty} a_n = \lim_{n \rightarrow \infty} b_n = 0$, that is $\forall p\exists n\forall m\ge n [|a_m| < \frac{1}{2^p} \;\wedge\;  |b_m| <\frac{1}{2^p}]$. \item (not using Brouwer's Continuity Principle):\\ 
$ \forall p\forall \zeta \in [\omega]^\omega \exists n[|a_{\zeta(n)}|<\frac{1}{2^p} \;\wedge\; |b_{\zeta(n)}| <\frac{1}{2^p}]$. \end{enumerate}   
   
  \end{lemma}
   
   \begin{proof} Define, for each $n$, $r_n :=\sqrt{(a_n)^2 +(b_n)^2}$. 
   
   \smallskip (i)   Let $p$ be given. Define $\alpha$ in $\mathcal{C}$ such that, for each $n$, if $\alpha(n)=0$, then $r_n<\frac{1}{2^p}$ and, if $\alpha(n) =1$, then $r_n>0$. Also find $y_0, y_1, \ldots$ in $\mathcal{R}$ such that, for each $n$, if $\alpha(n) = 1$, then for all $x$ in $[a,b]$,  $a_n \sin nx + b_n \cos nx = r_n \cos(nx  + y_n)$. 
   
Using the Continuity Principle, find $c,d,n$ such that $a<c<d<b$ and    $d-c>\frac{2\pi}{n}$ and  $\forall m \ge n\forall x \in[c,d][|r_m \cos(mx  + y_m)|<\frac{1}{2^p}]$. Note: $\forall m\ge n\exists x \in [c,d][\cos(mx + y_m)=1]$.  Conclude: $\forall m\ge n [r_m <\frac{1}{2^p}]$ and  $\forall m\ge n [|a_m| < \frac{1}{2^p} \;\wedge\;  |b_m| <\frac{1}{2^p}]$.  
    
    \smallskip
 (ii) Let $p$ be given. Define $\alpha$ in $\mathcal{C}$ such that, for each $n$, if $\alpha(n)=0$, then $r_n<\frac{1}{2^p}$ and, if $\alpha(n) =1$, then $r_n>0$. Also find $y_0, y_1, \ldots$ in $\mathcal{R}$ such that, for each $n$, if $\alpha(n) = 1$, then for all $x$ in $[a,b]$,  $a_n \sin nx + b_n \cos nx = r_n \cos(nx  + y_n)$. 
 
  Note that, for each $n$, for all $c,d$, if $d- c \ge\frac{2\pi}{n}$,  one may find $e,f$ such that $c<e<f<d$ and $f-e=\frac{2\pi}{3n}$ and, for all $x$ in $[e,f]$, $|\cos(nx+y_n)| \ge \frac{1}{2}$. 
    
  Assume $\zeta \in [\omega]^\omega$.   Define $\eta$ in $[\omega]^\omega$ such that $b-a>\frac{2\pi}{\zeta\circ\eta(0)}$, and, for each n, $\zeta\circ \eta(n+1) > 3\cdot\zeta\circ \eta(n)$.  Define an infinite sequence $c_0, d_0, c_1, d_1, \ldots$ of reals such that $c_0=a$ and $c_1 = b$ and, for all $n$, $c_n<c_{n+1}<d_{n+1}<d_n$ and $d_{n+1} -c_{n+1}= \frac{2\pi}{3\cdot\zeta\circ\eta(n)}$ and, if $\alpha\bigl(\zeta\circ\eta(n)\bigr)=1$, then, for all $x$ in $[c_{n+1}, d_{n+1}]$, $|\cos\bigl(\zeta\circ\eta(n)\cdot x+y_ {\zeta\circ\eta(n)}\bigr)|\ge \frac{1}{2}$. Find $x$ such that, for all $n>0$, $c_n < x < d_n$. Find $N$ such that for all $n\ge N$, if $\alpha(n) =1$, then $|r_n \cos (nx+y_n) |< \frac{1}{2^{p+1}}$.  Find $n$ such that $\zeta\circ\eta(n) \ge N$. Note: $|\cos \bigl(\zeta\circ\eta(n)\cdot x +y_{\zeta\circ\eta(n)}\bigr)|\ge\frac{1}{2}$ and  $r_{\zeta\circ\eta(n)} < \frac{1}{2^p}$. Define $m:=\zeta\circ\eta(n)$. Either $\alpha(m) = 0$ and $r_m <\frac{1}{2^p}$, or $\alpha(m) = 1$, and also: $r_m<\frac{1}{2^p}$. 
  
  Conclude: $\forall \zeta \in [\omega]^\omega\exists n[r_{\zeta(n)}<\frac{1}{2^p}]$ and  $\forall \zeta \in [\omega]^\omega\exists n[|a_n|<\frac{1}{2^p}\;\wedge\; |b_n|<\frac{1}{2^p}]$.  
   \end{proof}
   
   Lemma \ref{L:cl} was the subject of Cantor's first publication on trigonometric series, see \cite{cantor70a}. Cantor of course did not have Brouwer's Continuity Principle and proved Lemma \ref{L:cl}(ii). By classical logic, (ii) implies (i). He thus gives a classical, indirect proof of (i). Cantor's proof is complicated. I suspect that a direct constructive proof of (i), avoiding Brouwer's Continuity Principle, is impossible but I have no proof of this fact.  Such a proof would explain Cantor's obvious difficulty of finding an easy argument for (i).
   
   Note that Lemma \ref{L:cl}(i) enables us to simplify the proofs of Theorems  \ref{T:cu}, \ref{T:cucofinite} and \ref{T:evfullunique}. In the proofs of these theorems we are given a subset $\mathcal{X}$ of $[-\pi,\pi]$ such that, for some $a,b$, $a<b$ and $[a,b]\subseteq \mathcal{X}$, and   an infinite sequence $b_0, a_1, b_1, \dots $ of reals such that, for all $x$ in $\mathcal{X}$, $\frac{b_0}{2}+\sum_{n>0} a_n\sin nx + b_n \cos nx=0$. Lemma \ref{L:cl}(i) enables us to conclude: the sequence $b_0, a_1, b_1, \dots $ converges to $0$ and thus \textit{ is bounded}. The second halves of the proofs, using the Riemann-Kronecker suggestion, then become superfluous. Cantor in fact used   Lemma \ref{L:cl}(i) in this way. The constructive mathematician who does not want to use Brouwer's Continuity Principle still has to invoke the Riemann-Kronecker suggestion.
   
   In order to show that a similar observation applies to Theorem \ref{T:triumph} we prove an extension of Lemma \ref{L:cl}.
   
   We need another consequence of Brouwer's Continuity Principle:
  \begin{quote} {\it Let $a, b$ in $\mathcal{R}$ be given such that $a<b$ and let $\mathcal{X}$ be a co-enumerable subset of $[a,b]$.  Let $R$ be a \emph{real} subset of $\mathcal{X}\times\mathbb{N}$, that is: $\forall n \forall x \in [a,b]\forall y \in [a,b][(x=_\mathcal{R}y \;\wedge\; xRn )\rightarrow yRn]$.
  
If  $\forall x \in\mathcal{X} \exists n[x R n]$, then \\ $\forall x \in \mathcal{X} \exists c\exists d[a<c<x<d<b\;\wedge\; \exists n\forall y \in [c,d]\cap \mathcal{X}[ y Rn]]$.} \end{quote} 
 
 One may prove this using the fact that there exists a continuous surjection from $\mathcal{N}$ onto $\mathcal{X}$. 
    \begin{lemma}\label{L:cl2} Let $a, b$ be real numbers such that $a<b$. and let $\mathcal{X}$ be a co-enumerable subset of $[a,b]$ such that $\forall x \in \mathcal{X}[ \lim_{n \rightarrow \infty} a_n \sin nx + b_n \cos nx = 0]$.  Then: \begin{enumerate}[\upshape (i)] \item (using Brouwer's Continuity Principle): $\lim_{n \rightarrow \infty} a_n = \lim_{n \rightarrow \infty} b_n = 0$. \item (not using Brouwer's Continuity Principle):\\
$ \forall p\forall \zeta \in [\omega]^\omega \exists n[|a_{\zeta(n)}|<\frac{1}{2^p} \;\wedge\; |b_{\zeta(n)}| <\frac{1}{2^p}]$. \end{enumerate}

  \end{lemma}
   
   \begin{proof}  The proof is a slight adaptation of the proof of Lemma \ref{L:cl}. 
   Define, for each $n$, $r_n :=\sqrt{(a_n)^2 +(b_n)^2}$. 
   
   \smallskip (i)   Let $p$ be given. Define $\alpha$ in $\mathcal{C}$ such that, for each $n$, if $\alpha(n)=0$, then $r_n<\frac{1}{2^p}$ and, if $\alpha(n) =1$, then $r_n>0$. Also find $y_0, y_1, \ldots$ in $\mathcal{R}$ such that, for each $n$, if $\alpha(n) = 1$, then for all $x$ in $[a,b]$,  $a_n \sin nx + b_n \cos nx = r_n \cos(nx  + y_n)$. Find $c,d$ such that $a<c<d<b$ and $n$ such that    $d-c>\frac{2\pi}{n}$ and  $\forall m \ge n\forall x \in[c,d]\cap \mathcal{X}[|r_m \cos(mx  + y_m)|<\frac{1}{2^{p+1}}]$. Note: $\forall m\ge n\exists x \in [c,d]\cap\mathcal{X}[\cos(mx + y_m)>\frac{1}{2}]$.  Conclude: $\forall m\ge n [r_m <\frac{1}{2^p}]$ and  $\forall m\ge n [|a_m| < \frac{1}{2^p} \;\wedge\;  |b_m| <\frac{1}{2^p}]$.   
    
    \smallskip
 (ii) Let $p$ be given. Define $\alpha$ in $\mathcal{C}$ such that, for each $n$, if $\alpha(n)=0$, then $r_n<\frac{1}{2^p}$ and, if $\alpha(n) =1$, then $r_n>0$. Also find $y_0, y_1, \ldots$ in $\mathcal{R}$ such that, for each $n$, if $\alpha(n) = 1$, then for all $x$ in $[a,b]$,  $a_n \sin nx + b_n \cos nx = r_n \cos(nx  + y_n)$. 
 
  Note that, for each $n$, for all $c,d$, if $d\ge c +\frac{4\pi}{n}$,  one may find $e,f,g,h$ such that $c<e<f<g<h<d$ and $f-e=h-g=\frac{2\pi}{3n}$ and, for all $x$ in $[e,f]\cup[g,h]$, $|\cos(nx+y_n)| \ge \frac{1}{2}$.

   Find $x_0, x_1, x_2, \ldots$ in $[a,b]$ such that, for all $x$ in $[a,b]$, if $\forall n[x \;\#_\mathcal{R}\; x_n]$, then $x \in \mathcal{X}$.

  Assume $\zeta \in [\omega]^\omega$.   Define $\eta$ in $[\omega]^\omega$ such that $b-a>\frac{4\pi}{\zeta\circ\eta(0)}$, and, for each n, $\zeta\circ \eta(n+1) > 6\cdot\zeta\circ \eta(n)$.  Define an infinite sequence $c_0, d_0, c_1, d_1, \ldots$ of reals such that $c_0=a$ and $c_1 = b$ and, for all $n$, $c_n<c_{n+1}<d_{n+1}<d_n$ and $d_{n+1} -c_{n+1}= \frac{2\pi}{3\cdot\zeta\circ\eta(n)}$\footnote{Note: $d_{n+1}-c_{n+1} = \frac{4\pi}{6\cdot\zeta\circ\eta(n)}>\frac{4\pi}{\zeta\circ\eta(n+1)}$.} and: either $x_n<c_{n+1}$ or $d_{n+1}<x_n$, and,  if $\alpha\bigl(\zeta\circ\eta(n)\bigr)=1$, then, for all $x$ in $[c_{n+1}, d_{n+1}]$, $|\cos\bigl(\zeta\circ\eta(n)\cdot x+y_ {\zeta\circ\eta(n)}\bigr)|\ge \frac{1}{2}$. Find $x$ such that, for all $n>0$, $c_n < x < d_n$. Note: for all $n$, $x\;\#_\mathcal{R}\; x_n$, and: $x\in \mathcal{X}$.  Find $N$ such that for all $n\ge N$, if $\alpha(n) =1$, then $|r_n \cos (nx+y_n) |< \frac{1}{2^{p+1}}$.  Find $n$ such that $\zeta\circ\eta(n) \ge N$. Note: $|\cos \bigl(\zeta\circ\eta(n)\cdot x +y_{\zeta\circ\eta(n)}\bigr)|\ge\frac{1}{2}$ and  $r_{\zeta\circ\eta(n)} < \frac{1}{2^p}$. Define $m:=\zeta\circ\eta(n)$. Either $\alpha(m) = 0$ and $r_m <\frac{1}{2^p}$, or $\alpha(m) = 1$, and also: $r_m<\frac{1}{2^p}$. 
  
  Conclude: $\forall \zeta \in [\omega]^\omega\exists n[r_{\zeta(n)}<\frac{1}{2^p}]$ and  $\forall \zeta \in [\omega]^\omega\exists n[|a_n|<\frac{1}{2^p}\;\wedge\; |b_n|<\frac{1}{2^p}]$.  
   \end{proof}
   
   There is a further extension of Lemma \ref{L:cl2}: one may replace the condition: \textit{$\mathcal{X}$ is co-enumerable} by: \textit{$\mathcal{X}$ has positive (Brouwer-)Lebesgue measure}. We do not treat this  more general \textit{Cantor-Lebesgue Theorem}  as we have no application for it in this paper. In 1903,  Lebesgue proved the \textit{Riemann-Lebesgue Lemma}, see \cite{lebesgue03}, and the Cantor-Lebesgue Theorem follows easily, see \cite{bary64}, par. 64.
   
 \subsection{Cantor's Uniqueness Theorem becomes trivial}   Brouwer's Continuity Principle and the Fan Theorem together lead to the following conclusion:
 
\begin{quote}{\it Let $a,b$ in $\mathcal{R}$ be given such that $a<b$ and let $R$ be a \emph{real} subset of $[a,b]\times \mathbb{N}$, that is: $\forall n \forall x \in [a,b]\forall y \in [a,b][(x=_\mathcal{R}y \;\wedge\; xRn )\rightarrow yRn]$.
 
 If $\forall x \in [a,b]\exists n[xRn]$, then $\exists N\forall x \in [a,b]\exists n\le N[xRn]$.} \end{quote}
 
 It follows that an infinite sequence of functions that converges to $0$ everywhere in some closed interval does so uniformly:
 \begin{quote}
 {\it Let $a,b$ in $\mathcal{R}$ be given such that $a<b$ and let $f_0, f_1, \ldots$ be an infinite sequence of functions from $[a,b]$ to such that $\forall x \in [a,b] \forall p \exists n \forall m>n[|f_m(x)|<\frac{1}{2^p}]$. Then $\forall p \exists n\forall m>n \forall x \in [a,b]  [|f_m(x)|<\frac{1}{2^p}]$.}
 \end{quote}
 
 Now let $b_0, a_1, b_1,  \ldots$ in $\mathcal{R}$ be given such that, for all $x$ in $[-\pi,\pi]$, $F(x) :=\frac{b_0}{2}+\sum_{n>0} a_n \sin nx + b_n \cos nx =0.$ Define, for each $m>0$, $F_m(x) :=\frac{b_0}{2}+\sum_{n=1}^{m} a_n \sin nx + b_n \cos nx$.
 
 Note: for each $n$, for each $m>n$, $a_n = \frac{1}{2\pi}\int_{-\pi}^\pi F_m(x)\sin nxdx$ and: $0 =\frac{1}{2\pi} \int_{-\pi}^\pi F(x) \sin nx dx = \lim_{m\rightarrow \infty} \frac{1}{2\pi}\int_{-\pi}^\pi F_m(x)\sin nxdx =a_n$. Conclude: for each $n>0$, $a_n=0$ and, for a similar reason, $b_n =0$ and also $b_0=0$. 
 
 We thus obtain Cantor's conclusion very quickly. 
 
 In fact, this short route was shown to Cantor himself by \textit{`Herr Appell'} who apparently confused pointwise and uniform convergence, see \cite{cantor80} and \cite{cooke93}.
 
 \section{The two possibilities Cantor saw for closed sets}Let $\mathcal{G}$ be an open subset of $(-\pi,\pi)$.  We call $\mathcal{G}^\ast:=\bigcup Ext_\mathcal{G} := \{x\in(-\pi,\pi)|\exists \mathcal{H} \in Ext_\mathcal{G}[x \in \mathcal{H}]\}$ \textit{the co-perfect hull of $\mathcal{G}$}. We claim: $(\mathcal{G}^\ast)^+ = \mathcal{G}^\ast$. For suppose: $x \in (\mathcal{G}^\ast)^+$. Find $n$ in $\mathbb{N}$, $y \in (x-\frac{1}{2^n}, x+\frac{1}{2^n})$ such that, for all $z$ in $(x-\frac{1}{2^n}, x+\frac{1}{2^n})$, if $y\;\#_\mathcal{R}\;z$, then $z \in \mathcal{G}^\ast$. Find $m_0$ such that $x-\frac{1}{2^n} + \frac{1}{2^{m_0}} < y-\frac{1}{2^{m_0}}$ and $y+\frac{1}{2^{m_0}}<x+\frac{1}{2^n} - \frac{1}{2^{m_0}}$. Note: for each $p>m_0$, $[x-\frac{1}{2^n} + \frac{1}{2^{m_0}} , y-\frac{1}{2^{p}}] \subseteq \mathcal{G}^\ast$ and: $[y+\frac{1}{2^{p}}<x+\frac{1}{2^n} - \frac{1}{2^{m_0}}]\subseteq \mathcal{G}^\ast$. Conclude, using the Heine-Borel Theorem: for each $p$, there is a finite subset of $Ext_\mathcal{G}$ covering $[x-\frac{1}{2^n} + \frac{1}{2^{m_0}} , y-\frac{1}{2^{p}}] \cup[y+\frac{1}{2^{p}}<x+\frac{1}{2^n} - \frac{1}{2^{m_0}}]$. Using an Axiom of Countable Choice, find an infinite sequence $\mathcal{H}_0, \mathcal{H}_1, \ldots$ of elements of $Ext_\mathcal{G}$ such that $(x-\frac{1}{2^n} + \frac{1}{2^{m_0}} , y) \cup(y, x+\frac{1}{2^n} - \frac{1}{2^{m_0}})\subseteq\bigcup\limits_{p \in \mathbb{N}}[x-\frac{1}{2^n} + \frac{1}{2^{m_0}} , y-\frac{1}{2^{p}}] \cup[y+\frac{1}{2^{p}},x+\frac{1}{2^n} - \frac{1}{2^{m_0}}]\subseteq \bigcup\limits_{p\in \mathbb{N}} \mathcal{H}_p$ and conclude: $(x-\frac{1}{2^n} + \frac{1}{2^{m_0}} , x+\frac{1}{2^n} - \frac{1}{2^{m_0}})\subseteq (\bigcup\limits_{p\in \mathbb{N}} \mathcal{H}_p)^+$, and: $x \in \mathcal{G}^\ast$.

 We distinguish two cases. \begin{enumerate}[\upshape (i)] \item $(-\pi,\pi) \subseteq \mathcal{G}^\ast$. Note: $(-\pi, \pi) = \bigcup\limits_{n \in \mathbb{N}} [-\pi+\frac{1}{2^n}, \pi-\frac{1}{2^n}]$, and using again the Heine-Borel Theorem, find an infinite sequence $\mathcal{H}_0, \mathcal{H}_1, \ldots$ of elements of $Ext_\mathcal{G}$ such that $(-\pi,\pi)\subseteq \bigcup\limits_{n \in \mathbb{N}} \mathcal{H}_n$.  Conclude: $\mathcal{G}$ is eventually full, and, using Theorem \ref{T:openevfullalmostenumerable}(ii): $[-\pi,\pi]\setminus \mathcal{G}$ is almost-enumerable. \item  $\exists x \in (-\pi,\pi)[x \notin \mathcal{G}^\ast]$. Note: for all $x$ in $(-\pi,\pi)$, if $x \notin \mathcal{G}^\ast$, then $x \notin (\mathcal{G}^\ast)^+$, that is: $\forall n \neg \forall z \in (x-\frac{1}{2^n},x)\cup(x, x+\frac{1}{2^n})[z\in \mathcal{G}^\ast]$.  In a very weak sense, therefore, every point of $(-\pi,\pi)\setminus \mathcal{G}^\ast$ is a limit point of  $(-\pi,\pi)\setminus \mathcal{G}^\ast$. \end{enumerate}
 
 This is a pale version of \textit{Cantor's Main Theorem}, as Brouwer calls it: \textit{every closed subset of $(-\pi,\pi)$ either is at most countable or contains a perfect subset},  (and is, therefore, in Cantor's view, equivalent to the continuum). Constructively, we can not prove: $\forall x \in (-\pi,\pi)[x \in \mathcal{G}^\ast ]\;\vee\; 
\exists x \in (-\pi,\pi)[x \notin \mathcal{G}^\ast     ]$; we can't even prove: $\neg\neg\bigl(\forall x \in (-\pi,\pi)[x \in \mathcal{G}^\ast ]\;\vee\; 
\exists x \in (-\pi,\pi)[x \notin \mathcal{G}^\ast     ]\bigr)$. Results related to this observation may be found in \cite{gielen81} and \cite{burgess80}. 

\section{Brouwer's work on Cantor's Main Theorem} In \cite{brouwer19}, Brouwer introduces the notion of a \textit{`set'}, (more or less a located closed subset of $\mathcal{R}$),  that admits of \textit{`an inner deconstruction'} or is `\textit{deconstructible}'\footnote{This  term from modern philosophical discourse seems a more apt translation than the literal `destructible'.}.  What he means is that the set can be split, in some sense effectively, into the set of its limit points and the set of its isolated points. If it can, he hopes that the set of its limit points is `deconstructible' again, and so on. In order to treat the `and so on',  he first develops an intuitionistic theory of countable ordinals.

Unfortunately,  there are mistakes and obscurities, as Brouwer himself knew.\footnote{In \cite{brouwer52}, Brouwer, mentioning \cite{brouwer19}, states: \textit{Reading over these developments to-day, one finds that they are obsolete and in need of radical recasting.}} We refrain from a detailed commentary.
Instead,   we make a guess as to how Brouwer could have defined his notions, had he read this paper.

Let $\mathcal{G}$ be an open and co-located subset of $(-\pi,\pi)$. We say that $\mathcal{G}$  is \textit{effectively extendible} if $\mathcal{G}^+$ is co-located again. We say that $\mathcal{G}$ is \textit{hereditarily effectively extendible} if and only if every element of $Ext_\mathcal{G}$ is co-located and we say that $\mathcal{G}$ \textit{admits of an effective final extension} if, in addition, $[-\pi,\pi]\setminus\bigcup Ext_\mathcal{G}$ is located. Brouwer uses here the expression: $[-\pi,\pi]\setminus\mathcal{G}$ is \textit{`vollst\"andig abbrechbar', `completely deconstructible'}.

It is not so easy for a `set' in Brouwer's sense to be completely deconstructible and it is not so easy for an open subset $\mathcal{G}$ of $(-\pi,\pi)$ to have an effective final extension. These are, from a constructive point of view, very strong conditions. But if they are satisfied, one hopes to be able to take the decision one, emulating Cantor, would like to make:  Brouwer's  `set', (our $[-\pi,\pi]\setminus \mathcal{G}$), is either (in a reasonable sense) countable, or (in a reasonable sense) at least as big as the continuum.

Unfortunately, there is another difficulty. Let $\mathcal{G}$ be an open subset of $(-\pi,\pi)$ and let $\mathcal{F}= [-\pi,\pi]\setminus \mathcal{G}$ be its complement. From the point of view taken in this paper, the (weak) derivative set $D_w(\mathcal{F})$ of $\mathcal{F}$ should be defined as the \textit{complement of the co-derivative of $\mathcal{G}$}, that is $D_w(\mathcal{F}):= [-\pi,\pi]\setminus \mathcal{G}^+$.  $D_w(\mathcal{F})$ contains the set $D(\mathcal{F})$ of the limit points of $\mathcal{F}$ but \textit{does not necessarily coincide with $D(\mathcal{F})$}. If one knows: $D(\mathcal{F}) = \mathcal{F}$ then   one may define an injective map from Cantor space $\mathcal{C}$ into $\mathcal{F}$ but, if one knows: $\mathcal{F} = D_w(\mathcal{F})$, one may  be unable to do so.

     \end{document}